\documentclass{article}

\usepackage{graphicx}

\usepackage{amsmath, amssymb, amsthm, enumerate}

\usepackage{amsmath,amssymb,amsthm,enumerate,mathrsfs}
\newtheorem{theorem}{Theorem}[section]

\newtheorem{lemma}[theorem]{Lemma}
\newtheorem{proposition}[theorem]{Proposition}
\theoremstyle{definition}
\newtheorem{definition}[theorem]{Definition}

\newtheorem{assum}{Assumption}[section]

\newtheorem{remark}[theorem]{Remark}

\numberwithin{equation}{section}

\begin{document}
\makeatletter
\begin{center}
\Large{\bf Duality theory for optimistic bilevel optimization}
\end{center}\vspace{5mm}
\begin{center}

\textsc{Houria En-Naciri\footnote{Laboratoire LASMA, Department of Mathematics, Sidi Mohammed Ben Abdellah University, Morocco (email: \textsf{houria.ennaciri@edu.umi.ac.ma})}  Lahoussine Lafhim\footnote{Laboratoire LASMA, Department of Mathematics, Sidi Mohammed Ben Abdellah University, Morocco (email: \textsf{lahoussine.lafhim@usmba.ac.ma})}   Alain  Zemkoho\footnote{School of Mathematical Sciences, University of Southampton, UK (email: \textsf{a.b.zemkoho@soton.ac.uk}). This author's work is supported by the EPSRC grant EP/V049038/1 and the Alan Turing Institute under the EPSRC grant EP/N510129/1}}\end{center}

\vspace{2mm}

\footnotesize{
\noindent\begin{minipage}{14cm}
{\bf Abstract:}
In this paper, we exploit the so-called value function reformulation of the bilevel optimization problem to develop duality results for the problem. Our approach builds on Fenchel-Lagrange-type duality to establish suitable results for the bilevel optimization problem. First, we overview some standard duality results to show that they are not applicable to our problem. Secondly, via the concept of partial calmness, we establish  weak and strong duality results. In particular, Lagrange, Fenchel-Lagrange, and Toland-Fenchel-Lagrange duality concepts are investigated for this type of problems under some suitable conditions. Thirdly, based on the use of some regularization of our bilevel program, we  establish sufficient conditions ensuring strong  duality results under a generalized Slater-type condition without convexity assumptions and without the partial calmness condition. Finally, without the Slater condition, a strong duality result is constructed for  the bilevel optimization problem with geometric constraint.
\end{minipage}
 \\[5mm]


\hbox to14cm{\hrulefill}\par

\section{Introduction}
In this paper, we are concerned with the following standard optimistic bilevel programming problem
\begin{equation}\label{calm0}\tag{P}
{\displaystyle \min_{x, y}} \ F(x, y) \;\;\mbox{s.t.} \;\; G_{i}\left( x,y\right) \leq 0, \; \; i\in I_{k}, \;\; y\in \mathcal{S}(x),
\end{equation}
where, for each $x\in \mathbb{R}^{n}$, $\mathcal{S}(x)$ represents the optimal solution set of the lower level problem

\begin{equation}\label{llp-1}\tag{L[$x$]}
{\displaystyle \min_{y}} \ f(x, y) \;\; \mbox{ s.t. } \;\; y\in K\left(x\right),
\end{equation}
$F :\mathbb{R}^n \times \mathbb{R}^m \rightarrow \mathbb{R}\cup \{+\infty\}$ and $f :\mathbb{R}^n \times \mathbb{R}^m \rightarrow \mathbb{R}\cup \{+\infty\}$ are the upper and lower level objective functions, while $G_i: \mathbb{R}^{n} \times \mathbb{R}^{m} \rightarrow \mathbb{R}\cup \{+\infty\}$ for $i\in I_{k} :=\{1,\ldots, k\}$ ($k\in \mathbb{N}\setminus\{0\}$) denotes the upper level constraint functions. Throughout the paper, we confine ourselves to the case where the lower level feasible set is defined, for all $x\in \mathbb{R}^n$, by
\begin{equation*}
K(x):=\left\{y\in \mathbb{R}^{m} \mid g_{j}(x, y) \leq 0, \; \; j\in I_{p} :=\{1,\cdots,p \}\right\}
\end{equation*}
with $g_j: \mathbb{R}^{n} \times \mathbb{R}^{m} \rightarrow \mathbb{R}\cup \{+\infty\}$ for $j=1, \ldots, p$ ($p\in \mathbb{N}\setminus\{0\}$).  The objective functions $F$ and $f$, as well as the constraint functions $G_i$ for $i=1, \ldots, k$ and $g_j$ for $j=1, \ldots, p$, are assumed to be proper convex functions. In the sequel, from \cite{zali}, we adopt the following standard operations involving $+\infty$ and $-\infty$:
\begin{equation*}
	\begin{array}{rll}
	\left( +\infty\right) -\left( +\infty\right)= \left( -\infty\right)-\left( -\infty\right)=\left( +\infty\right)+\left( -\infty\right)=\left( -\infty\right)+\left( +\infty\right) &=& +\infty,\\
	0 \cdot \left( +\infty\right) &=& +\infty,\\
 0 \cdot \left( -\infty\right) &=&0.
		\end{array}
\end{equation*}

Our main goal in this paper is to study the duality theory for the bilevel optimization problem \eqref{calm0}, which corresponds to the optimistic version of the problem; more details on bilevel optimization and related properties can be found in \cite{DempeFoundations2002,DempeZemkohoBook}, where surveys on different classes on the problem, including the pessimistic version, are also given.

The literature on duality theory for general optimization problems is very rich. For example, strong duality for nearly convex optimization problems is established by Bo\c{T} and Wanka \cite{ref10}, where three kinds of conjugate dual problems are devoted to the primal optimization problem; i.e., the Lagrange, Fenchel, and  Fenchel-Lagrange dual problems. Also, Mart\'{\i}nez-Legaz and Volle \cite{ref44} introduced in particular a duality scheme for the problem of minimizing the difference of two extended real-valued convex functions (i.e., d.c. functions for short) under finitely many d.c. constraints in terms of the Legendre-Fenchel conjugates of the involved convex functions. Another related interesting problem in conjugate duality theory developed by  Rockafellar (see \cite{ref11} and \cite{ref12}) is the well-known perturbational approach, consisting in the use of a perturbation function as the keystone to obtain a dual problem for a general primal one by means of the Fenchel conjugation. Both problems satisfy weak duality (in the sense that the optimal value of the dual problem is less than or equal to the optimal value of the primal one), as a direct consequence of the Fenchel–Young inequality, whereas conditions ensuring strong duality (no duality gap and dual problem solvable) can be found in these references and many other ones in the literature.

However, only a small number of papers have been dedicated to duality theory for bilevel optimization problems. In \cite{refext}, Aboussoror and Adly adapted the Fenchel–Lagrange duality, which has been first introduced by Bo\c{t} and Wanka in \cite{ref bot} and tailored to convex programming problems, showing that a strong duality result holds between a bilevel nonlinear optimization problem with extremal-value function and its Fenchel-Lagrange dual under appropriate assumptions. Afterwards, a duality approach for the strong bilevel programming problem \eqref{calm0} via a regularization and the Fenchel-Lagrange duality was presented by Aboussoror, Adly and Saissi in \cite{refextend}. More precisely, using $\varepsilon-$approximate solutions of the lower level problem, they consider a regularized problem  (\ref{calm0}$_{\epsilon}$) of \eqref{calm0}. Via this regularization and the operation of scalarization they present in \cite{ref semi vect}, a duality approach using conjugacy for a semivectorial bilevel programming problem with  the upper and lower levels being vectorial and scalar respectively, are proposed.

In this paper, we overview the literature to construct appropriate duality results for bilevel optimization problems, including Fenchel, Lagrange-Fenchel, and Toland-Lagrange-based techniques. Clearly, the aim of this paper is two-fold. We explore standard duality results from the literature and show their limitations, in terms of their applicability to the bilevel optimization problem. Secondly, we explore a number of approaches to develop tractable approaches to obtain duality results for the bilevel optimization problem. In the latter case, we first use an exact penalization of problem \eqref{calm0} based on the well-know concept of  partial calmness (see \cite{ye}) to completely characterize Fenchel-Lagrange-type results. The key to this exact penalization is that the resulting problem can be viewed as a minimization problem of a d.c. objective function under convex constraints. Note that the technique used to get such a transformation is inspired by the work of \cite{ref+,ref44}, and conjugacy is the main tool involved in this technique. The results are obtained under the Slater constraint qualification and further regularity conditions and some closedness conditions. Additionally, based on the use of a regularization of problem \eqref{calm0}, we establish sufficient conditions ensuring strong duality results under a generalized Slater-type condition without convexity assumptions and without the partial calmness condition. Finally, without any  Slater-type condition, strong duality results are investigated for the bilevel optimization problem with geometric constraint.

In next section, we first recall some definitions of convex analysis tools needed in this work. The preliminaries are mainly concerned with relevant properties of the subdifferential, also of the epigraph and some conjugate properties of convex functions. In Section $3$, we collect some existing results while showing their limitations for our program. In Section $4$, we investigate duality theory for problem \eqref{calm0} using the partial calmness under some appropriate conditions. Subsequently, we completely characterize the strong duality for problem \eqref{calm0} in the presence of partial calmness. In Section $5$ however, in the absence of partial calmness, but under a generalized Slater constraint qualification, we provide a strong duality for the nonconvex bilevel optimization problem \eqref{calm0}. Then, the special case where the bilevel problem is geometrically constrained  is also studied. Finally, we conclude our work in Section $6$.

\section{Preliminaries}

Here, we recall some fundamental definitions and results essentially related to convex mappings. To proceed, let $D$ be a subset of $\mathbb{R}^{n}$. The sets $int D$, $cl D$ and $co D$, stand for the topological interior, the closure and the convex hull of $D$, respectively. Moreover, the convex cone (containing the origin of $\mathbb{R}^{n}$) generated by $D$ is denoted by $cone D$, and its closure is $\overline{cone} D$, the positive polar cone of $D$ is given by
\[
D^{\circ}=\left\{x^{*} \in \mathbb{R}^{n}:\left\langle x^{*}, x\right\rangle \geqslant 0, \ \forall x \in D\right\},
\]
while the indicator function $\delta_{D}: \mathbb{R}^{n} \rightarrow \mathbb{R}\cup\{+\infty\}$ of $D$ is defined by
\[
\delta_{D}(x)=\left\{\begin{array}{ll}
0, & \text { if } x \in D, \\
+\infty, & \text { if } x \notin D.
\end{array}\right.
\]
Additionally, the support function $\sigma_{D}: \mathbb{R}^{n} \rightarrow \bar{\mathbb{R}}=\mathbb{R}\cup\{-\infty\}\cup\{+\infty\}$ of $D$ is
$$
\sigma_{D}\left(x^{*}\right)=\sup _{x \in D}\left\langle x^{*}, x\right\rangle .
$$
For the optimization problem \eqref{calm0} we denote by ${\mathcal V}\left( \ref{calm0}\right) $ its optimal objective value and this notation is extended to the optimization problems that we use in this paper.

For a given extend real valued function $f: \mathbb{R}^{n} \rightarrow \bar{\mathbb{R}}$, the epigraph, the lower semicontinuous hull, and the effective domain of $f$ are respectively defined by
\begin{equation*}
\begin{array}{rll}
\operatorname{epi} f & = & \left\{(x, r) \in \mathbb{R}^{n} \times \mathbb{R}\mid f(x) \leqslant r\right\},\\[1ex]
\operatorname{epi}(\operatorname{cl}f) & = &\operatorname{cl}(\operatorname{epi}f),\\[1ex]
\operatorname{dom}f & = &\left\{x \in \mathbb{R}^{n} \mid f(x)<+\infty\right\}.
\end{array}
\end{equation*}

We say that $f$ is proper if $f(x)>-\infty$ for all $x \in \mathbb{R}^{n}$, and $\operatorname{dom}f \neq \emptyset$.
\begin{definition}
Let $f: \mathbb{R}^{n} \rightarrow \bar{\mathbb{R}}$  be a function, and $D$ a nonempty subset of $\mathbb{R}^{n}$.
 The conjugate function of $f$ relative to the set $D$ is denoted by $f_{D}^{*}$ and defined on $\mathbb{R}^{n}$ by
\begin{equation*}\label{eq.dual}
f_{D}^{*}(p)=\sup _{x \in D}\left\{\langle p, x\rangle-f(x)\right\},
\end{equation*}
where $\langle., .\rangle$, denotes the inner product of two vectors in $\mathbb{R}^{n}$, i.e., for given $x=\left(x_{1}, \ldots, x_{n}\right)^{T}$ and $y=\left(y_{1}, \ldots, y_{n}\right)^{T}$, we have $\langle x, y\rangle=\displaystyle\sum_{i=1}^{n} x_{i} y_{i} .$ When $D=\mathbb{R}^{n},$
we get the usual Legendre-Fenchel conjugate function of $f$ denoted by $f^{*}$.
\end{definition}
 Let $D \subset \mathbb{R}^{n}$ and $\bar x\in cl D$.  The contingent cone to $D$ at $\bar{x}$ is the set
 $$
 T\left(\bar{x},D\right) =\left\{y \in \mathbb{R}^{n} \mid \exists t_{k} \downarrow 0, \exists y_{k} \rightarrow y, \bar{x}+t_{k} y_{k} \in D \ \forall k\in \mathbb{N}\right\}.
 $$
\begin{definition}
 Let $f: \mathbb{R}^{n} \rightarrow \mathbb{R} \cup\{+\infty\}$ be a proper convex function and $\bar{x} \in \operatorname{dom}f$.
  The subdifferential in the sense of convex analysis of $f$ at $\bar{x}$ denoted by $\partial f(\bar{x})$ is the set
 $$
 \partial f(\bar{x})=\left\{x^{*} \in \mathbb{R}^{n}\mid  f(x) \geq f(\bar{x})+\left\langle x^{*}, x-\bar{x}\right\rangle \ \forall x \in \mathbb{R}^{n}\right\}.
 $$
 An element $x^{*} \in \partial f(\bar{x})$ is called a subgradient of $f$ at $\bar{x}$.
\end{definition}
\begin{remark}\label{eq.}
Let $f: \mathbb{R}^{n} \rightarrow \mathbb{R} \cup\{+\infty\}$. Then, the following assertions hold:
\begin{itemize}
\item[(i)] $x^{*} \in \partial f(\bar{x})\ \text{if and only if}\ \left\langle x^{*}, \bar{x}\right\rangle=f(\bar{x})+f^{*}\left(x^{*}\right)$;
\item[(ii)]  $f(x)+f^{*}\left(x^{*}\right) \geq\left\langle x^{*}, x\right\rangle \ \ \  \forall x, x^{*} \in \mathbb{R}^{n}$ (the Young-Fenchel inequality);
\item[(iii)] If $\operatorname{cl}f$ is proper and convex, then it holds that $f^{**}=\operatorname{cl}f$.
\end{itemize}
\end{remark}
\begin{proposition}\label{inf-sum}
 Let $f_{1}, \ldots, f_{k}: \mathbb{R}^{n} \rightarrow \mathbb{R} \cup\{+\infty\}$ be proper convex functions. If the set $\displaystyle\bigcap_{i=1}^{k} \operatorname{ri}\left(\operatorname{dom}\left(f_{i}\right)\right)$ is nonempty, then
$$
\left(\sum_{i=1}^{k} f_{i}\right)^{*}(p)=\left(f_{1}^{*} \square \ldots \square f_{k}^{*}\right)(p),
$$
where for a nonempty subset $B$ of $\mathbb{R}^{n}$, $\text{ri} B$ denotes the relative interior of $B$, i.e., the interior
		of $B$ relative to the smallest affine set containing $B$, and
		$$\left(f_{1}^{*} \square \ldots \square f_{k}^{*}\right)(p):=\inf \left\{\sum_{i=1}^{k} f_{i}^{*}\left(p_{i}\right): p=\sum_{i=1}^{k} p_{i}\right\}$$
is the infimal convolution of $f_{1}^{*}, \ldots , f_{k}^{*}$, and for each $p \in \mathbb{R}^{n}$ the infimum is attained.
\end{proposition}
\begin{proposition}\label{cl2}
Let $I$ be an index set, and let $\left\{f_{i}: i \in I\right\}$ be a family of proper convex functions, where $f_{i}:\mathbb{R}^{n} \rightarrow \mathbb{R} \cup\{+\infty\}$. Then the following statements hold:
 \begin{itemize}
 \item[(i)] $\operatorname{epi}\left(\underset{i \in I}{\sup}f_{i}\right)=\displaystyle\underset{i \in I}{\bigcap} \operatorname{epi}f_{i}$;
 \item[(ii)]  $\left(\underset{i \in I}{\inf}  f_{i}\right)^{*}=\underset{i \in I}{\sup} f_{i}^{*}$ and hence, $\operatorname{epi}\left(\underset{i \in I}{\inf } f_{i}\right)^{*}=\displaystyle\underset{i \in I}{\bigcap}  \operatorname{epi} f_{i}^{*}$;
 \item[(iii)] If the set $I$ is finite and  $\displaystyle\bigcap_{i=1}^{k} \operatorname{ri}\left(\operatorname{dom}\left(f_{i}\right)\right)$ is nonempty, then
 $$
 \operatorname{epi}\left(\left(\sum_{i=1}^{k} f_{i}\right)^{*}\right)=\sum_{i=1}^{k} \operatorname{epi}\left(f_{i}^{*}\right) .
 $$
 \end{itemize}
\end{proposition}
 Finally, any point $p \in \mathbb{R}^{n}$ can be considered as a function on $\mathbb{R}^{n}$ in such a way that $p(x):=\langle p, x\rangle$, for any $x \in \mathbb{R}^{n} .$ Thus, for any $\alpha \in \mathbb{R}$ and any function $h: \mathbb{R}^{n}\rightarrow \bar{R}$,
 \begin{equation*}\label{eqr}
 \left( h+p+\alpha\right)^{\ast}\left( x^{\ast}\right)=h^{\ast}\left( x^{\ast}-p\right)-\alpha , \ \ \text{for all} \ x^{\ast}\in\mathbb{R}^{n}
  \end{equation*}
  and
 \begin{equation*}\label{eqr-1}
 \operatorname{epi}(h+p+\alpha)^{*}=\operatorname{epi} h^{*}+(p,-\alpha) .
 \end{equation*}
For convenience, let $\theta\in\mathbb{R}^{m}$ and $M: \mathbb{R}^{n}\rightarrow \mathbb{R}^{m}$, the composition mapping $\theta \circ K$ will be denoted by $\theta M$.
\section{Limitations of some existing theory}
Consider the following lower level value reformulation of the bilevel optimization problem \eqref{calm0} under consideration:
\begin{equation}\label{llvf}\tag{LLVF}
\left\{\begin{array}{ll}
{\displaystyle \min_{x, y}} \ F(x, y) \\
G_{i}(x,y) \leq 0 \quad & \forall i \in I_{k}=\{1,\dots,k\}, \\
g_{j}(x, y) \leq 0 & \forall j \in I_{p}=\{1,\dots,p\}, \\
f(x, y)-\varphi(x) \leq 0,
\end{array}\right.
\end{equation}
where $\varphi$ denotes the optimal value function of the lower level problem \eqref{llp-1}:
\[
\varphi(x):=\underset{y}\min~ \left\{ f(x,y)| \;\; y\in K(x)\right\}.
\]
\subsection{Convex optimization with only inequality constraints}
In this subsection, we review some duality results for the mathematical programming problem
\begin{equation}\label{w-o-ineq}\tag{P$_{in}$}
\min \{\phi \left( x\right)|\; \psi\left( x\right) \leq 0\}
\end{equation}
with $\phi:\mathbb{R}^{n}\rightarrow\mathbb{R}$ and $\psi =\left( \psi_{1},\cdots ,\psi_{p}\right)^{T}:\mathbb{R}^{n} \rightarrow\mathbb{R}^{r}$, where $\psi\left( x\right) \leq 0$ means that $\psi_{i}\left( x\right) \leq 0$ for all $i\in I_{r}=\{1,\cdots ,r\}$.
To define the dual problem of \eqref{w-o-ineq}, Bot et al. \cite{ref bot} considered a perturbation by using the following perturbation function:
\begin{equation*}
\Phi\left( x,\alpha ,\beta\right) =
\left\{
\begin{array}{lcl}
\phi\left( x+\alpha\right) & \ \ \ & \text{if} \ \psi\left( x\right) \leq \beta, \\
+\infty & \ \ \ & \text{otherwise}.
\end{array}
\right.
\end{equation*}
Afterwards, they calculate the conjugate of $\Phi$ and obtained the following Fenchel-Lagrange dual problem
\begin{equation}\label{w-o-ineq-dual}\tag{D$_{in}$}
{\displaystyle \sup_{\alpha^{\ast}\in\mathbb{R}^{n}, \ \beta^{\ast}\in\mathbb{R}_{+}^{r} } } \ \bigg\{ -\phi^{\ast} \left( \alpha^{\ast}\right) + {\displaystyle \inf_{x\in\mathbb{R}^{n}} } \ \left[ \langle \alpha^{\ast} ,x\rangle + \langle\beta^{\ast},\psi\left( x\right)\rangle \right]  \bigg\}.
\end{equation}
In general, to deal with strong duality which charaterizes \eqref{w-o-ineq}, we need to use a Slater-type constraint qualification.
\begin{definition}
We say that the genaralized Slater constraint qualification holds for the familly $\left( \psi_{i}\right)_{i\in I_{r}}$ if there exist an element $\bar{x}\in \mathbb{R}^{n}$ such that $\psi_{i}\left(\bar{x}\right) < 0$ for $i\in N$ and $\psi_{i}\left(\bar{x}\right) \leq 0$ for $i\in L$, where $N$ and $L$ are  respectively given as
\begin{equation*}
\begin{array}{ll}
N=\{i\in I_{r} \ |& \text{such that} \ \psi_{i} \ \text{is not an affine function} \},\\
L=\{i\in I_{r} \ |& \text{such that} \ \psi_{i} \ \text{is an affine function} \}. 
\end{array}
\end{equation*}
\end{definition}
\begin{theorem}[\cite{ref bot}]\label{f-l-convex}
Let $\phi$ and $\psi_{i}$, with $i\in I_{r}$ be convex functions. Assume that the genaralized Slater constraint qualification holds for the familly $\left( \psi_{i}\right)_{i\in I_{r}}$. If ${\mathcal V} \eqref{w-o-ineq}$ is finite, then strong duality holds and ${\mathcal V} \eqref{w-o-ineq} ={\mathcal V} \eqref{w-o-ineq-dual}$.
\end{theorem}
Note that in \cite{ref bot} the Slater constraint qualification is used as a sufficient condition to establish strong Fenchel-Lagrange duality.		
Unfortunately, due to the presence of $f(x, y)-\varphi(x)$ as a constraint, problem \eqref{llvf}does not satisfy the Slater constraint qualification. Moreover, problem \eqref{llvf}is not convex, even if the data are convex. Hence, Theorem \ref{f-l-convex} is not applicable for problem \eqref{llvf}. However, there is a chance that some of  these requirements can hold for the simple bilevel optimization problem (see, e.g., \cite{Shehu}), where the lower level problem is an unperturbed optimization problem.

\subsection{Convex optimization with inequality and equality constraints}
Next, we add some equality constraints to problem \eqref{w-o-ineq} to get the following primal optimization problem with equality constraints
\begin{equation}\label{w-o-eq}\tag{P$_{eq}$}
\min \{\phi \left( x\right) \ | \ \psi\left( x\right) \leq 0, \ \ \omega\left( x\right) =0\},
\end{equation}
where $\omega :\mathbb{R}^{n}\rightarrow\mathbb{R}^{q}$ is an affine-linear mapping and $\phi :\mathbb{R}^{n}\rightarrow\mathbb{R}$ and $\psi :\mathbb{R}^{n}\rightarrow\mathbb{R}^{r}$ are two functions such that the function $\left( \phi ,\psi \right) :\mathbb{R}^{n}\rightarrow\mathbb{R}\times\mathbb{R}^{r}$ defined by $\left( \phi ,\psi \right)\left( x\right) =\left( \phi\left( x\right)  ,\psi\left( x\right)  \right)$ is convex-like with respect to the cone $\mathbb{R}_{+}\times\mathbb{R}_{+}^{r}$; that is, the set $\left( \phi ,\psi \right)\left( \mathbb{R}^{n}\right) +\left(\mathbb{R}_{+}\times\mathbb{R}_{+}^{r}\right) $ is convex.
The Lagrange dual \eqref{w-o-eq-dual} of \eqref{w-o-eq} is defined as 
\begin{equation}\label{w-o-eq-dual}\tag{D$_{eq}$}
{\displaystyle \sup_{\beta^{\ast}\in\mathbb{R}_{+}^{r}, \ \gamma^{\ast}\in\mathbb{R}^{q} } } \ {\displaystyle \inf_{x\in\mathbb{R}^{n}} } \  \bigg\{ \phi \left(x\right) + \langle \beta^{\ast}, \psi\left( x\right) \rangle+ \langle \gamma^{\ast}, \omega\left( x\right) \rangle  \bigg\}.
\end{equation}
In \cite{refrev}, Bot et al. provide a strong duality theorem for \eqref{w-o-eq} and its Lagrange  dual problem by using some regularity conditions expressed by means of the contingent cone. More precisely, they work under the assumption that $\bar{x}$ is an optimal solution of the primal problem \eqref{w-o-eq}. To clarify their concept, let us introduce the following sets:
\begin{equation*}
\begin{array}{lcl}
\Sigma & = & \{x\in\mathbb{R}^{n} \ | \ \psi\left( x\right) \leq 0, \ \ \omega\left( x\right) =0 \}, \\
{\mathcal E} & = & \left( \phi\left( \bar{x}\right)-\phi\left( x\right) -\sigma , -\psi\left( x\right) -z,-\omega\left( x\right)\right) \ | \ x\in \mathbb{R}^{n}, \sigma \geq 0, \ z\in \mathbb{R}_{+}^{r} \}.
\end{array}
\end{equation*}
We arrive now to the strong duality theorem stated in the latter reference.
\begin{theorem}\label{f-l-convex-eq}
	If $\bar{x}$ is an optimal solution for \eqref{w-o-eq}, the following assertions are equivalent
	\begin{enumerate}
		\item The equality $T_{\mathcal E}\left(0,0_{\mathbb{R}^{r}},0_{\mathbb{R}^{q}} \right) \bigcap \left( \mathbb{R}_{+}^{\ast}\times 0_{\mathbb{R}^{r}}\times0_{\mathbb{R}^{q}}\right) =\emptyset$ holds;
		\item ${\mathcal V} \eqref{w-o-eq}={\mathcal V} \eqref{w-o-eq-dual}$ and there exists a point $(\bar{\beta}, \bar{\gamma}) \in \mathbb{R}_{+}^{r}\times \mathbb{R}^{q}$ that solves problem \eqref{w-o-eq-dual}.
		
		In this case, we have $\langle \bar{\beta} , \psi\left( \bar{x}\right) \rangle =0$.
	\end{enumerate}
\end{theorem}

We now  analyze the applicability of this theorem to our bilevel program \eqref{calm0}. To proceed, we consider the classical optimistic bilevel programming problem	with linear lower level problem
\begin{equation}\label{GoodEx}
\left\{\begin{array}{ll}
{\displaystyle \min_{x, y}} \ F(x, y) \\
	G_{i}(x,y) \leq 0 \quad  \forall i \in I_{k}=\{1,\dots,k\}, \\
    y\in S\left( x\right)=\arg\underset{y}{\min}~\left\{C^{T}y\,|\;\; Ax+By\leq d\right\},
\end{array}\right.
\end{equation}
where $d\in\mathbb{R}^{r}$, \ $C \in\mathbb{R}^{m}$, \ $A \in\mathbb{R}^{r}\times\mathbb{R}^{n}$ and $B \in\mathbb{R}^{r}\times\mathbb{R}^{m}$. Its lower level value function reformulation can be written as
\begin{equation*}
{\displaystyle \min_{x, y}} \ F(x, y) \ : \ G_{i}(x,y) \leq 0 \quad  \forall i \in I_{k}=\{1,\dots,k\}, \ Ax+By\leq d, \ C^{T}y-\varphi\left( x\right) = 0,
\end{equation*}
where the lower-level value function $\varphi$ is given by
\begin{equation*}
\varphi\left( x\right) ={\displaystyle \min_{y}} \ \{ C^{T}y\,| \;\; Ax+By\leq d\}.	
\end{equation*}
In general, even if the lower level problem is fully affine linear (i.e., in $(x,y)$), as it is the case here, the optimal value function $\varphi$ is not linear. This function is typically only be piecewise linear. Consequently, Theorem \ref{f-l-convex-eq} will rarely be applicable to a problem of the form \eqref{GoodEx}.

\subsection{Nonconvex optimization with inequality and equality constraints}
Now, we do not assume that \eqref{w-o-eq} has a solution, also no convexity assumption is imposed. Let us consider the Lagrangian
\begin{equation*}
L\left(\varsigma ,\beta ,\gamma , x\right)=\varsigma \phi\left( x\right)+\langle \beta ,\psi\left( x\right) \rangle +\langle \gamma ,\omega\left( x\right) \rangle.
\end{equation*}
It can be seen that for all $\varsigma \in\mathbb{R}_{+}$, $\beta\in\mathbb{R}_{+}^{r}$ and $\gamma\in\mathbb{R}^{q}$, it holds that
\begin{equation*}
	{\displaystyle \inf_{x\in\mathbb{R}^{n}}} \ L\left(\varsigma ,\beta ,\gamma , x\right)\leq  \varsigma  {\mathcal V}\left( P_{eq}\right), \ \ \ \forall \varsigma \geq 0.
\end{equation*}
Consequently, weak duality is satisfied in this case.

In \cite{fabin-12}, a strong Lagrange duality theorem for the nonconvex optimization problem \eqref{w-o-eq} and its dual is established under a suitable generalized Slater assumption. Based on this reference, it is said that the generalized Slater condition holds if
\begin{equation}\label{cq-w-p-casnorm}
\overline{cone} \ \bigg( \left(\psi ,\omega\right) \left( \mathbb{R}^{n}\right) + \mathbb{R}^{r}_{+}\times \{0_{\mathbb{R}^{q}}\}\bigg)= \mathbb{R}^{r}\times\mathbb{R}^{q}.
\end{equation}
To proceed we need the following construction
\begin{equation*}
	\Psi \left( \mathbb{R}^{n}\right)=\bigg \{\left(\phi\left( x\right) ,\psi\left( x\right) ,\omega\left( x\right)  \right)\in \mathbb{R}\times\mathbb{R}^{r}\times\mathbb{R}^{q} \ | \ x\in \mathbb{R}^{n} \bigg \}.
\end{equation*}
By assuming that a genaralized Slater condition holds together with an additional condition, the authors prove the following strong duality result.
\begin{theorem}\label{fin-nonconvex}
	Consider problems \eqref{w-o-eq} and \eqref{w-o-eq-dual} and suppose that
	\begin{itemize}
		\item ${\mathcal V} \eqref{w-o-eq}$ is finite,
		\item $int \bigg ( \text{co} \left( \Psi \left( \mathbb{R}^{n}\right) \right)+ \left( \mathbb{R}_{+}\times\mathbb{R}_{+}^{r}\times\mathbb{R}^{q}\right) \bigg) \neq\emptyset$,
		\item the generalized Slater condition $\left( \ref{cq-w-p-casnorm}\right)$ holds.
	\end{itemize}
Furthermore, let one of the following assumption hold:
 \begin{description}
 	\item[$A\left( 1\right) $] $\text{cone} \ \bigg ( int \bigg [ \text{co} \left( \Psi \left( \mathbb{R}^{n}\right) \right)-{\mathcal V} \eqref{w-o-eq} \left( 1,0_{\mathbb{R}_{+}\times\mathbb{R}_{+}^{r}\times\mathbb{R}^{q}}\right)  + \left( \mathbb{R}_{+}\times\mathbb{R}_{+}^{r}\times\mathbb{R}^{q}\right) \bigg] \bigg)$ is pointed,
 	\item[$A\left( 2\right) $]
 	$\left( 0,0\right) \notin int \bigg [ \text{co} \left( \Psi \left( \mathbb{R}^{n}\right) \right)-{\mathcal V} \eqref{w-o-eq} \left( 1,0_{\mathbb{R}_{+}\times\mathbb{R}_{+}^{r}\times\mathbb{R}^{q}}\right)  + \left( \mathbb{R}_{+}\times\mathbb{R}_{+}^{r}\times\mathbb{R}^{q}\right) \bigg]$.
 \end{description}
Then, $\mathcal{V} \eqref{w-o-eq}=\mathcal{V} \eqref{w-o-eq-dual}$ and there exists $\left(\bar{\beta},\bar{\gamma}\right)\in \mathbb{R}^{r}_{+}\times\mathbb{R}^{q}$ such that
	\begin{equation*}
		{\mathcal V} \eqref{w-o-eq}=	{\displaystyle \inf_{x \in\mathbb{R}^{n}}} \ L\left(1,\bar{\beta} ,\bar{\gamma} ,x\right).
	\end{equation*}
\end{theorem}

Next, we do not assume that \eqref{calm0} has a solution, also no convexity assumption is imposed. Let us consider the Lagrangian
\begin{equation*}
L\left(\varsigma ,\alpha ,\beta ,\gamma , x,y \right)=\varsigma F(x,y)+\gamma(f(x, y)-\varphi(x)))+\sum_{i=1}^{k}\alpha_{i} G_{i}(x,y)+\sum_{j=1}^{p}\beta_{j} g_{j}(x,y) .
\end{equation*}
It can be seen that for all $\gamma\in\mathbb{R}$, $\alpha\in\mathbb{R}_{+}^{k}$ and $\beta\in\mathbb{R}_{+}^{p}$
\begin{equation*}
	{\displaystyle \inf_{\left( x,y\right) \in\mathbb{R}^{n}\times\mathbb{R}^{m}}} \ L\left(\varsigma ,\alpha ,\beta ,\gamma , x,y \right)\leq  \varsigma  {\mathcal V}\left( {\mathcal P}\right), \ \ \ \forall \varsigma \geq 0.
\end{equation*}
Consequently, the weak duality holds. To provide the strong duality, we need that the following Slater-type condition holds
\begin{equation}\label{cq-w-p-cLLVF}
\overline{cone} \ \bigg( \left( f-\varphi,G,g\right) \left( \mathbb{R}^{n}\times\mathbb{R}^{m}\right)  + \{0_{\mathbb{R}}\}\times\mathbb{R}_{+}^{k+p}\bigg)= \mathbb{R}\times\mathbb{R}^{k+p}.
\end{equation}
Unfortunately, this is not the case, as  any vector $\left( a,0_{\mathbb{R}^{k}},0_{\mathbb{R}^{p}}\right) $ with $a < 0$ does not belong to the set in the left-hand-side of equation $\left( \ref{cq-w-p-cLLVF}\right)$. Hence, Theorem \ref{fin-nonconvex} in not applicable for \eqref{calm0} problem.

 \section{Duality under partial calmness}
To deal with the fact that most standard constraint qualifications do not hold for  problem \eqref{calm0}, we are going to use the partial calmness concept in this section. To proceed, we use the reformulation \eqref{llvf}. Here, when $\mathcal{S}\left( x\right) =\emptyset$ for some $x\in \mathbb{R}^{n}$, we use the convention $\inf \emptyset = \infty$.

In general, the value function $\varphi$ is not lower semicontinuous, even when the involved functions satisfy this property. To ensure the validity of the lower semicontinuity of $\varphi$, we will impose some additional conditions on the set-valued map $K$; see \cite[Lemma 3.2]{GauvinDubeau1982} for relevant details. To proceed here, we consider the perturbed version of problem \eqref{llvf} linearly parameterized by $t\in \mathbb{R}$:
\begin{equation}\label{calm1}\tag{P[$t$]}
\left\{
\begin{array}{l}
{\displaystyle \min_{x,y}} \ \ F\left( x,y\right)\\
f\left( x,y\right)-\varphi \left( x\right)+t=0,\\
G_{i}\left( x,y\right)\leq 0 \ \ \text{for all} \ i\in I_{k}=\{1,\cdots ,k\}, \\
g_{j}\left( x,y\right)\leq 0 \ \ \text{for all} \ j\in I_{p}=\{1,\cdots ,p\}.
\end{array}
\right.
\end{equation}
\begin{definition}\cite{ye}
	we say that the unperturbed problem \eqref{llvf} is partially calm
	at its feasible solution $\left( \bar{x},\bar{y}\right)$ if there is a constant $\lambda > 0$ and a neighborhood $U$ of
	the triple $\left( \bar{x},\bar{y},0\right)\in {\mathbb{R}}^{n}\times {\mathbb{R}}^{m}\times \mathbb{R}$ such that
	\begin{equation}\label{calm1100}
	F\left( x,y\right)-F\left( \bar{x},\bar{y}\right)+\lambda \mid t\mid \geq 0 \ \ \ \text{for all} \ \ \left( x,y,t\right)\in U           \ \ \text{feasible to} \ \eqref{calm1}.
	\end{equation}
\end{definition}
In the next lemma, we show that the partial calmness to the bilevel program \eqref{llvf} allows us an exact penalization of these problem. In fact, the upper-level objective in our problem may be neither Lipschitz nor continuous, conditions that are usually imposed in existing results; see, e.g., \cite{ye}. However, we formulate another proof here when the upper-level objective function is only lower semi-continuous.
\begin{lemma}\label{calm2}
	Let $\left( \bar{x},\bar{y}\right)$ be a partially calm local optimal solution to the bilevel program \eqref{llvf}, and let the
	upper-level objective function $F$ be lower semi-continuous at this point. Then $\left( \bar{x},\bar{y}\right)$ is a local optimal solution to the penalized problem
	\begin{equation}\label{calm3}\tag{P$^{\lambda}$}
	\left\{
	\begin{array}{l}
	{\displaystyle \min_{x,y}} \ \ \lambda^{-1}F\left( x,y\right)+ f\left( x,y\right)-\varphi\left( x\right) \\
	G_{i}\left( x,y\right)\leq 0 \ \ \text{for all} \ i\in I_{k}=\{1,\cdots ,k\}, \\
	g_{j}\left( x,y\right)\leq 0 \ \ \text{for all} \ j\in I_{p}=\{1,\cdots ,p\},
	\end{array}
	\right.
	\end{equation}
	where $\lambda > 0$ is the constant from the partial calmness condition $\left( \ref{calm1100}\right)$.
\end{lemma}
\begin{proof}
	Suppose that $\left( \bar{x},\bar{y}\right)$ is a partially calm
	local optimal solution to the bilevel program  \eqref{llvf} whereas  $\left( \bar{x},\bar{y}\right)$ is not a local solution of \eqref{calm3}. First, we can find a parameter $\lambda > 0$ and a neighborhood $U$ of
	the triple $\left( \bar{x},\bar{y},0\right)\in {\mathbb{R}}^{n}\times {\mathbb{R}}^{m}\times \mathbb{R}$ such that
	\begin{equation}\label{sci1}
	F\left( x,y\right)-F\left( \bar{x},\bar{y}\right)+\lambda \mid t\mid \geq 0 \ \ \ \text{for all} \ \ \left( x,y,t\right)\in U           \ \ \text{feasible to} \eqref{calm1}.
	\end{equation}
	Secondly, there is a  sequence $\left( x_{k},y_{k}\right)\in \mathbb{R}^{n}\times\mathbb{R}^{m}$ feasible to \eqref{calm3} and  satisfying  $\left( x_{k},y_{k}\right)\rightarrow  \left( \bar{x},\bar{y}\right)$ such that
	\begin{equation}\label{sci2}
	F\left( x_{k},y_{k}\right)+ \lambda \left( f\left( x_{k},y_{k}\right)-\varphi\left( x_{k}\right) \right) < F\left( \bar{x},\bar{y}\right)
	\end{equation}
	for all $k\in \mathbb{N}$. It follows from the lower semi-continuity of $F$ at $\left( \bar{x},\bar{y}\right)$ that
	there is $\delta > 0$ and $\eta > 0$ such that $V=\left[ \left( \bar{x},\bar{y}\right)+ \eta \mathbb{B}_{\mathbb{R}^{n}\times \mathbb{R}^{m}}\right] \times \left(  -\delta , \delta \right)  \subset U$ and that the inequality
	\begin{equation}\label{sci3}
	-\lambda \delta \leq F\left( x,y\right)- F\left( \bar{x},\bar{y}\right)
	\end{equation}
	holds true for all $\left( x,y\right)\in \left( \bar{x},\bar{y}\right)+ \eta \mathbb{B}_{\mathbb{R}^{n}\times \mathbb{R}^{m}}$. Setting $t_{k}=f\left( x_{k},y_{k}\right)-\varphi\left( x_{k}\right)$, two cases have to be considered.\\
	\hspace*{.5cm}Case I. Let $\left( \left( x_{k},y_{k}\right), t_{k}\right) \in V$; then combining
	$\left( \ref{sci1}\right)$ and $\left( \ref{sci2}\right)$, it follows that
	\begin{equation*}
		0\leq F\left( x_{k},y_{k}\right)-F\left( \bar{x},\bar{y}\right)+ \lambda \left( f\left( x_{k},y_{k}\right)-\varphi\left( x_{k}\right) \right) < 0.
	\end{equation*}
	However, this is a contradiction.\\
	\hspace*{.5cm}Case II. Assume that $\left( \left( x_{k},y_{k}\right), t_{k}\right) \notin V$. On the one side, since for $k$ large enough $\left( x_{k},y_{k}\right)\in \left( \bar{x},\bar{y}\right)+ \eta \mathbb{B}_{\mathbb{R}^{n}\times \mathbb{R}^{m}}$, one has $t_{k}\notin \left(  -\delta , \delta \right)$. On the other side, by definition of optimal value function, we have $t_{k}\geq 0$. Consequently, $t_{k}\geq \delta$. Hence, by $\left( \ref{sci3}\right)$,
	\begin{equation}
	-\lambda \left( f\left( x_{k},y_{k}\right)-\varphi\left( x_{k}\right)\right)  \leq F\left( x,y\right)- F\left( \bar{x},\bar{y}\right).
	\end{equation}
	Combining the latter with $\left( \ref{sci2}\right)$, one obtains
	\begin{equation*}
		0\leq F\left( x_{k},y_{k}\right)-F\left( \bar{x},\bar{y}\right)+ \lambda \left( f\left( x_{k},y_{k}\right)-\varphi\left( x_{k}\right) \right) < 0.
	\end{equation*}
	However, this inequality is a contradiction. Since both cases lead to a contradiction, we conclude that $\left( \bar{x},\bar{y}\right)$ is a local optimal solution to the penalized problem \eqref{calm3}.
\end{proof}

If further assumptions are made about the marginal map $\varphi$ and the lower level objective function, the converse result may be obtained.
\begin{lemma}
	Let $\left( \bar{x},\bar{y}\right)$ with $\bar{y}\in {\mathcal S}\left( \bar{x}\right)$ be a local optimal solution of the penalized problem \eqref{calm3} for some $\lambda >0$. Assume that $f$ and $\varphi$ are locally Lipschitz near $\left( \bar{x},\bar{y}\right)$ with Lipschitz modulus $k_{f}$ and $k_{\varphi}$, respectively. Then, $\left( \bar{x},\bar{y}\right)$ is a partially calm local optimal solution to the bilevel optimization problem \eqref{llvf}.
\end{lemma}
\begin{proof}
	Suppose that there exists $\lambda > 0$ such that $\left( \bar{x},\bar{y}\right)$ solves \eqref{calm3} locally. Then, there exists $\delta > 0$ such that we have the inequality
	\begin{equation}\label{sci4}
	\lambda^{-1}F\left( \bar{x},\bar{y}\right) \leq \lambda^{-1} F\left( x,y\right)+f\left( x,y\right)-\varphi\left( x\right).
	\end{equation}
	for all $\left( x,y \right)\in \left( \bar{x},\bar{y}\right) + \delta \mathbb{B}_{\mathbb{R}^{n}\times \mathbb{R}^{m}}$ feasible for problem \eqref{calm3}.
	Let us set $\eta =\left\{\frac{\delta}{2k_{f}};\frac{\delta}{2k_{\varphi}};\frac{\delta}{2}\right\}$. Then, for any point $\left( x,y\right)\in \left( \bar{x},\bar{y}\right) + \eta \mathbb{B}_{\mathbb{R}^{n}\times \mathbb{R}^{m}}$ feasible  for problem \eqref{calm3}, one has
	\begin{equation*}
	\begin{array}{lcl}
	\mid f\left( x,y\right)-\varphi\left( x\right) \mid & = & \mid f\left( x,y\right)-f\left( \bar{x},\bar{y}\right)+\varphi\left( \bar{x}\right) - \varphi\left( x\right) \mid \\
	& \leq & k_{f}\parallel \left(x,y\right)- \left( \bar{x},\bar{y}\right)\parallel + k_{\varphi}\parallel x-\bar{x}\parallel\\
	& \leq & \delta .
	\end{array}
	\end{equation*}
	Thus, setting $t= \varphi\left( x\right)-f\left( x,y\right)$, the local optimality $\left( \ref{sci4}\right)$ leads to
	\begin{equation*}
		F\left( x,y\right)-F\left( \bar{x},\bar{y}\right)+\lambda \mid t\mid \geq 0.
	\end{equation*}
	The latter inequality  is valid for all $\left( x,y\right)\in \bigg (\left( \bar{x},\bar{y}\right) + \eta \mathbb{B}\bigg )\times \left( -\delta ;\delta\right)$  feasible to \eqref{calm1}. Consequently, $\left( \bar{x},\bar{y}\right)$ is a partially calm local optimal solution to the bilevel program \eqref{llvf}.
\end{proof}

\subsection{Duality under the Slater condition}
In this subsection, we aim to describe the dual problem of the auxiliary problem \eqref{calm3} and characterize the corresponding assumption. To proceed, let the feasible set of problem \eqref{calm3} be denoted by
\begin{equation}\label{eq4}
\mathcal{S} =\Big\{(x,y)\in\mathbb{R}^{n}\times\mathbb{R}^{m}: G_{i}(x,y)\leqslant 0, \ i\in I_{k}, \ \  g_{j}(x,y)\leqslant 0, \ j\in I_{p} \Big\},
\end{equation}
and consider the following dual of problem \eqref{calm3}:
\begin{equation}\label{weak-problem-stand}\tag{D$^{\lambda}$}
\begin{array}{l}
\underset{x^{*} \in \operatorname{dom} \varphi^{*}}{\inf} \underset{ \underset{ \alpha\in\mathbb{R}_{+}^{k}, \beta\in\mathbb{R}_{+}^{p},}
	{\left( z^{*},q^{\ast}\right) \in \mathbb{R}^{n}\times\mathbb{R}^{m}}
}{\sup}   \bigg\{ \varphi^{*}\left(x^{*}\right)-(\lambda^{-1}F+ f)^{*}\left( z^{\ast}, q^{\ast}\right)\\
\qquad \qquad \qquad \qquad \qquad \qquad -\left.\left( \displaystyle\sum_{i=1}^{k} \ \alpha_{i} G_{i} +{\displaystyle \sum_{j=1}^{p}} \  \beta_{j}g_{j}\right)^{\ast}\left(x^{\ast}-z^{*},-q^{\ast}\right)\right\}.
\end{array}
\end{equation}

Next, we will characterize the weak and strong duality between \eqref{calm3} and \eqref{weak-problem-stand}. To proceed, and from here on, we denote by $G=\left(G_{1},\cdots ,G_{k}\right)^{T}$ and $g=\left(g_{1},\cdots ,g_{p} \right)^{T}$ the vectors of upper and lower constraint functions, respectively.
 \begin{theorem}\label{weak-slater-op}
 	Consider problem \eqref{calm3} for some $\lambda > 0$ and let $F,f, g_{1}, \ldots, g_{p},G_{1}, \ldots, G_{k}$ be extended real-valued convex functions on $\mathbb{R}^{n}\times\mathbb{R}^{m}$. Assume that $\varphi$ is lower semicontinuous. Then, the weak duality between \eqref{calm3} and \eqref{weak-problem-stand} holds. Thus is, ${\mathcal V} \eqref{calm3} \geq {\mathcal V} \eqref{weak-problem-stand}$.
 \end{theorem}
 \begin{proof}
 	Since $\varphi$ is lower semicontinuous, then the equality $\varphi^{**}(x)=\varphi(x)$ holds. Following the standard convexification technique, and using the Fenchel equality for $\varphi(\cdot )$, stated in $(\ref{eq.dual})$, the problem \eqref{calm3} takes the following form
 	\begin{equation}\label{eq6}\tag{P$^{\lambda}$}
 	\underset{ x^{*} \in \operatorname{dom}  \varphi^{*} } {\inf} \ \ \underset{ \underset{G\left( x,y\right) \leq 0, \ g\left( x,y\right) \leq 0} {(x,y) \in \mathbb{R}^{n}\times\mathbb{R}^{m}}} { \inf} \bigg\{ \lambda^{-1} F(x,y)+ f(x,y)+ \varphi^{*}\left(x^{*}\right)-\left\langle x^{*}, x\right\rangle\bigg\}.
 	\end{equation}
 	For any $x^{*}\in\operatorname{dom}\varphi^{*}$ consider
 	\begin{equation}\label{parametre}\tag{P$^{\lambda}\left( x^{\ast}\right) $}
 	\underset{ \underset{G\left( x,y\right) \leq 0, \ g\left( x,y\right) \leq 0} {(x,y) \in \mathbb{R}^{n}\times\mathbb{R}^{m}}} { \inf} \bigg\{ \lambda^{-1} F(x,y)+ f(x,y)+ \varphi^{*}\left(x^{*}\right)-\left\langle x^{*}, x\right\rangle\bigg\},
 	\end{equation}
 	and set $\phi \left( x,y\right) =\lambda^{-1} F(x,y)+ f(x,y)+ \varphi^{*}\left(x^{*}\right)-\left\langle x^{*}, x\right\rangle$. Since, \eqref{parametre} is a convex optimisation problem, its Fenchel-Lagrange duality reformulation is
 	\begin{equation}\label{w-calm-1}\tag{D$^{\lambda}\left(x^{\ast}\right)$}
 	\begin{array}{ll}
 	\underset{ \underset{ \left( s^{*},q^{\ast}\right) \in\mathbb{R}^{n}\times\mathbb{R}^{m} }
 		{\alpha\in\mathbb{R}_{+}^{k}, \beta\in\mathbb{R}_{+}^{p}}
 	}{\sup} &  \bigg\{ -\phi^{\ast}\left( s^{*}, q^{\ast}\right)- \left( \displaystyle\sum_{i=1}^{k} \ \alpha_{i} G_{i}+{\displaystyle \sum_{j=1}^{p}} \  \beta_{j} g_{j}\right)^{\ast}\left(-s^{*},-q^{*}\right)  \bigg\}.
 	\end{array}
 	\end{equation}
 	
 		On the other hand, by an easy calculation, using the equality stated in $\left( \ref{eqr}\right)$ together with the definition of conjugate function, we get
 	\begin{equation}\label{w-calm-3}
 	\phi^{\ast}\left( s^{*}, q^{\ast}\right)= \left( \lambda^{-1} F+ f\right)^{\ast}\left( s^{*}+x^{\ast}, q^{\ast}\right)-\varphi^{\ast}\left( x^{\ast}\right).
 	\end{equation}
 	Inserting $\left(  \ref{w-calm-3}\right)$ in \eqref{w-calm-1}, we obtain
 	\begin{equation}\tag{D$^{\lambda}\left(x^{\ast}\right)$}
 	\begin{array}{l}
 	\underset{ \underset{ \alpha\in\mathbb{R}_{+}^{k}, \beta\in\mathbb{R}_{+}^{p},  }
 		{\left( s^{*},q^{\ast}\right) \in \mathbb{R}^{n}\times\mathbb{R}^{m}}
 	}{\sup}   \left\{ \varphi^{*}\left(x^{*}\right)-(\lambda^{-1}F+ f)^{*}\left( s^{*}+x^{\ast}, q^{\ast}\right)\right.\\
\qquad\quad \qquad \qquad \qquad \quad -\left.\left( \displaystyle\sum_{i=1}^{k} \ \alpha_{i} G_{i} +{\displaystyle \sum_{j=1}^{p}} \  \beta_{j}g_{j}\right)^{\ast}\left(-s^{*},-q^{\ast}\right)\right\}.
 	\end{array}
 	\end{equation}
 	Setting $z^{\ast}=s^{\ast}+x^{\ast}$, it follows
 	\begin{equation}\tag{D$^{\lambda}\left(x^{\ast}\right)$}
 	\begin{array}{l}
 	\underset{ \underset{ \alpha\in\mathbb{R}_{+}^{k}, \beta\in\mathbb{R}_{+}^{p},}
 		{\left( z^{*},q^{\ast}\right) \in \mathbb{R}^{n}\times\mathbb{R}^{m}}
 	}{\sup}   \left\{ \varphi^{*}\left(x^{*}\right)-(\lambda^{-1}F+ f)^{*}\left( z^{\ast}, q^{\ast}\right)\right.\\
\qquad \quad \qquad \qquad \qquad \quad -\left.\left( \displaystyle\sum_{i=1}^{k} \ \alpha_{i} G_{i} + {\displaystyle \sum_{j=1}^{p}} \  \beta_{j}g_{j}\right)^{\ast}\left(x^{\ast}-z^{*},-q^{\ast}\right)\right\}.
 	\end{array}
 	\end{equation}
 	Taking the infimum of \eqref{w-calm-1} for $x^{*}\in\operatorname{dom} \varphi^{*}$, we arrived at the desired dual problem \eqref{weak-problem-stand}. Finally, from \cite[Proposition 2.1]{ref10}, we have ${\mathcal V} \eqref{calm3} \geq {\mathcal V} \eqref{weak-problem-stand}$.
 \end{proof}

In general, the strong duality between \eqref{calm3} and \eqref{weak-problem-stand} does not necessarily hold. To overcome this situation, we introduce the following slater CQ:
 \begin{equation}\label{cq-slater-lambda}
\left\{
   \begin{array}{l}
    \exists (\bar{x},\bar{y}) \in \text{dom} \ \left( \lambda^{-1}F +f\right) \cap {\displaystyle \bigcap^{k}_{i=1}} \ \text{ri} \left( \text{dom} G_{i}\right) \cap {\displaystyle \bigcap^{p}_{j=1}} \ \text{ri} \left( \text{dom} g_{j}\right) \, \text{such that} \\
   g_{j}(\bar{x},\bar{y})<0, \ \ \forall j=1,\dots,p,\\
   G_{i}(\bar{x},\bar{y})<0, \ \ \forall i=1,\dots,k.
 \end{array}
  \right.
  \end{equation}
The following result provides a strong duality result for problem \eqref{calm3}.
 \begin{theorem}
 Let $F$, $f$, $g_{1}$, \ldots, $g_{p}$, and $G_{1}$, \ldots, $G_{k}$ be extended real-valued convex functions on $\mathbb{R}^{n}\times\mathbb{R}^{m}$. Moreover, suppose that CQ $\left( \ref{cq-slater-lambda}\right)$ is satisfied  and $\varphi$ is lower semicontinuous. Then, the strong duality between \eqref{calm3} and \eqref{weak-problem-stand} holds; that is, ${\mathcal V} \eqref{calm3}={\mathcal V} \eqref{weak-problem-stand}$ and the dual problem \eqref{weak-problem-stand} has an optimal solution.
 \end{theorem}
\begin{proof} Let $x^{\ast}\in \text{dom} \ \varphi^{\ast}$ and consider the following problem
\begin{equation}\label{strong-wc-p}\tag{P$^{\lambda}\left(x^{\ast} \right)$}
\left\{
\begin{array}{ll}
\inf \ \lambda^{-1}F(x, y)+ f(x, y)- \langle x,x^{\ast}\rangle +\varphi^{\ast}(x^{\ast}) & \\
\left( x,y\right) \in \mathbb{R}^{n}\times\mathbb{R}^{m} &\\
G_{i}(x,y)\leqslant 0 \quad \quad i=1,\dots,k, &\\
g_{j}(x,y)\leqslant 0 \quad \quad j=1,\dots,p.
\end{array}
\right.
\end{equation}
Since, $\lambda^{-1}F(x, y)+ f(x, y)- \langle x,x^{\ast}\rangle +\varphi^{\ast}(x)$, $G_{i}$, with $i=1,\dots,k$ and $g_{j}$ with $j=1,\dots,p$, are convex and proper we can introduce the Lagrange dual problem to \eqref{strong-wc-p} with $\alpha =\left( \alpha_{1},\cdots ,\alpha_{k}\right)\in \mathbb{R}_{+}^{k}$ and $\beta =\left( \beta_{1},\cdots ,\beta_{k}\right)\in \mathbb{R}_{+}^{p}$ as dual variables
\begin{equation}\label{strong-wc-d}\tag{D$^{\lambda}\left(x^{\ast} \right)$}
\begin{array}{ll}
\underset{\left( \alpha ,\beta\right) \in \mathbb{R}_{+}^{k}\times\mathbb{R}_{+}^{p}}{\sup}\underset{\left( x,y\right) \in \mathbb{R}^{n}\times\mathbb{R}^{m}}{\inf} &  \bigg\{ \lambda^{-1}F(x, y)+ f(x, y)- \langle x,x^{\ast}\rangle \\
&
+\varphi^{\ast}(x^{\ast})+\displaystyle\sum_{i=1}^{k} \alpha_{i} G_{i}\left(x,y\right) +\sum_{j=1}^{p} \beta_{j}g_{j}\left(x,y\right)  \bigg\}.
\end{array}
\end{equation}
Since \eqref{cq-slater-lambda} condition holds, it follows from \cite[Theorem 28.2]{rock1} that the optimal objective values of \eqref{strong-wc-p} and its Lagrange dual \eqref{strong-wc-d} are equal and, moreover, there exists an optimal solution $\bar{\alpha} =\left( \bar{\alpha}_{1},\cdots ,\bar{\alpha}_{k}\right)\in \mathbb{R}_{+}^{k}$ and $\bar{\beta} =\left( \bar{\beta}_{1},\cdots ,\bar{\beta}_{k}\right)\in \mathbb{R}_{+}^{p}$ of \eqref{strong-wc-d} such that
\begin{equation*}
	\begin{array}{lcl}
	{\mathcal V} \eqref{strong-wc-p} & = &
		\underset{\left( \alpha ,\beta\right) \in \mathbb{R}_{+}^{k}\times\mathbb{R}_{+}^{p}}{\sup}\underset{\left( x,y\right) \in \mathbb{R}^{n}\times\mathbb{R}^{m}}{\inf}  \bigg\{ \lambda^{-1}F(x, y)+ f(x, y)- \langle x,x^{\ast}\rangle \\
		& & +\varphi^{\ast}(x^{\ast})+\displaystyle\sum_{i=1}^{k} \alpha_{i} G_{i}\left(x,y\right)+\sum_{j=1}^{p} \beta_{j}g_{j}\left(x,y\right)  \bigg\} \\
    	 & = & \underset{\left( x,y\right) \in \mathbb{R}^{n}\times\mathbb{R}^{m}}{\inf}  \left\{ \lambda^{-1}F(x, y)+ f(x, y)\right.\\
    & & \qquad \qquad \qquad - \left.\langle x,x^{\ast}\rangle +\varphi^{\ast}(x)+\displaystyle\sum_{i=1}^{k} \bar{\alpha}_{i} G_{i}\left(x,y\right)+\sum_{j=1}^{p} \bar{\beta}_{j}g_{j}\left(x,y\right)\right\}.
	\end{array}
\end{equation*}
Since, $\text{dom} \ \left( \lambda^{-1}F+ f- \langle \cdot ,x^{\ast}\rangle +\varphi^{\ast}(x^{\ast})\right) =dom \left(\lambda^{-1}F+f\right)$, we have
$${\displaystyle \bigcap^{k}_{i=1}} \ \text{ri} \ dom \left(G_{i}\right)\cap {\displaystyle \bigcap^{p}_{j=1}} \ \text{ri} \ dom \left(g_{j}\right) \cap \text{ri} \ dom \left(\lambda^{-1}F+f\right) \neq\emptyset .$$
Hence, one can deduce from Proposition \ref{inf-sum} that
\begin{equation*}
	\begin{array}{lcl}
		{\mathcal V} \eqref{strong-wc-p} & = & -\bigg( \lambda^{-1}F+ f- x^{\ast} +\varphi^{\ast}(x^{\ast})+\displaystyle\sum_{i=1}^{k} \bar{\alpha}_{i} G_{i}+\sum_{j=1}^{p} \bar{\beta}_{j}g_{j}  \bigg)^{\ast}\left( 0\right) \\
		& = & \underset{ \left( s^{\ast} ,q^{\ast}\right) \in \mathbb{R}^{n}\times\mathbb{R}^{m}} {\sup} \bigg\{ \varphi^{*}\left(x^{*}\right)-(\lambda^{-1}F+ f)^{*}\left( s^{*}+x^{\ast},q^{\ast}\right)\\
       &   & \qquad \qquad \qquad \quad-\left.\left( {\displaystyle\sum_{i=1}^{k}} \  \bar{\alpha}_{i} G_{i} +{\displaystyle \sum_{j=1}^{p}}  \bar{\beta}_{j}g_{j}\right)^{\ast}\left(-s^{\ast},-q^{\ast}\right)\right\}
	\end{array}
\end{equation*}
and that there exists $\bar{s}^{\ast}\in\mathbb{R}^{n}$ and $\bar{q}^{\ast}\in\mathbb{R}^{m}$ such that the supremum is attained. Thus
\begin{equation*}
{\mathcal V} \eqref{strong-wc-p} = \varphi^{*}\left(x^{*}\right)-(\lambda^{-1}F+ f)^{*}\left( \bar{s}^{\ast}+x^{\ast},\bar{q}^{\ast}\right)-\left( {\displaystyle\sum_{i=1}^{k}} \  \bar{\alpha}_{i} G_{i} +{\displaystyle \sum_{j=1}^{p}}  \bar{\beta}_{j}g_{j}\right)^{\ast}\left(-\bar{s}^{\ast},-\bar{q}^{\ast}\right).
\end{equation*}

Setting $\bar{z}^{\ast}=\bar{s}^{\ast}+x^{\ast}$, we get
\begin{equation*}
{\mathcal V} \eqref{strong-wc-p} = \varphi^{*}\left(x^{*}\right)-(\lambda^{-1}F+ f)^{*}\left( \bar{z}^{\ast},\bar{q}^{\ast}\right)-\left( {\displaystyle\sum_{i=1}^{k}} \  \bar{\alpha}_{i} G_{i} +{\displaystyle \sum_{j=1}^{p}}  \bar{\beta}_{j}g_{j}\right)^{\ast}\left(x^{\ast}-\bar{z}^{\ast},-\bar{q}^{\ast}\right).
\end{equation*}
Hence, ${\mathcal V} \eqref{strong-wc-p}={\mathcal V} \eqref{strong-wc-d}$. Since, $x^{\ast}$ is arbitrary choosen in $\text{dom} \ \varphi^{\ast}$, then ${\mathcal V} \eqref{calm3}={\mathcal V} \eqref{weak-problem-stand}$. Thus, the proof is
complete.
\end{proof}

 \begin{remark}
 To compute the conjugate function of the value fucntion $\varphi$, we asume that $K\left( x\right)$ is independent of $x$. Thus, we set for all $x\in\mathbb{R}^{n}$,  $K\left( x\right) =K=\{y\in\mathbb{R}^{m}|g\left( y\right) \leq 0\}$. Then
 	\begin{equation*}
 		\varphi^{*}(x^{*}) = (f+\delta_{\mathbb{R}^{n}\times K})^{*}(0,x^{*})
 	\end{equation*}
and we have
 	$$
 	\begin{aligned}
 	\varphi^{*}(x^{*}) &=\underset{x\in\mathbb{R}^{n}}{\sup} \bigg\{\langle x^{*}, x\rangle-\varphi(x): x \in \mathbb{R}^{n}\bigg\} \\
 	&=\underset{x\in\mathbb{R}^{n}}{\sup}\bigg\{\langle x^{*}, x\rangle-\underset{y\in\mathbb{R}^{m}}{\inf}\Big\{  f(x, y)+\delta_{K}(y)\Big\}\bigg\} \\
 	&=\underset{x\in\mathbb{R}^{n}}{\sup} \ \underset{y\in\mathbb{R}^{m}}{\sup}\bigg\{\langle x^{*}, x\rangle-f(x,y)-\delta_{K}(y)\bigg\}\\
 	&=\underset{x,y}{\sup} \bigg\{\langle x^{*}, x\rangle-(f(x, y)+\delta_{K}(y))\bigg\} \\
 	& =(f+\delta_{\mathbb{R}^{n}\times K})^{*}(0,x^{*}).
 	\end{aligned}
 	$$
 \end{remark}

\subsection{Duality under further regularity conditions}
In general, weak duality (hence, strong duality) between \eqref{calm3} and \eqref{weak-problem-stand} does not necessarily hold in many situations, including when $\varphi$ is not lower semicontinuous. 
In order to describe the weak and strong dualities between \eqref{calm3} and \eqref{weak-problem-stand}, we modify the structure of the dual problem as
\begin{equation}\label{form-3}\tag{D$_{m}^{\lambda}$}
\underset{x^{*} \in \operatorname{dom} \varphi^{*}}{\inf} \underset{\underset{\alpha\in\mathbb{R}_{+}^{k}, \beta\in\mathbb{R}_{+}^{p}, \left( z^{*},q^{\ast}\right) \in\operatorname{dom}F^{*}\cap\operatorname{dom}f^{*}}{(u_{i}^{*},v_{i}^{*})\in\operatorname{dom}G_{i}^{*}, (u_{j}^{*},v_{j}^{*})\in\operatorname{dom}g_{j}^{*}} }{\sup}   \bigg\{ \psi \left(x^{\ast},z^{\ast},q^{\ast}, u^{\ast},v^{\ast},\alpha ,\beta\right) -\delta^{\ast}_{\mathbb{R}^{n}\times\mathbb{R}^{m}} \left( s^{\ast},t^{\ast}\right) \bigg\},
\end{equation}
where $\psi$ is given by
\begin{equation*}
\psi \left(x^{\ast},z^{\ast},q^{\ast}, u^{\ast},v^{\ast},\alpha ,\beta\right) =\varphi^{*}\left(x^{*}\right)-(\lambda^{-1}F+ f)^{*}\left( z^{*},q^{\ast}\right)-\displaystyle\sum_{i=1}^{k} \alpha_{i} G_{i}^{*}\left(u_{i}^{*},v_{i}^{*}\right)-\sum_{j=1}^{p} \beta_{j}g_{j}^{*}\left(u_{j}^{*},v_{j}^{*}\right)
\end{equation*}
with
$u^{\ast}=\left( u^{\ast}_{1},\cdots ,u^{\ast}_{k+p}\right)$, $v^{\ast}=\left( v^{\ast}_{1},\cdots ,v^{\ast}_{k+p}\right)  $, $s^{\ast}= x^{\ast}- z^{*}-\displaystyle\sum_{i=1}^{k} \alpha_{i} u_{i}^{*}-\sum_{j=1}^{p} \beta_{j}u_{j}^{*}$, \\
and $t^{\ast}=-q^{\ast}-\displaystyle\sum_{i=1}^{k} \alpha_{i} v_{i}^{*}-\sum_{j=1}^{p} \beta_{j} v_{j}^{*}$.

Secondly, we exploit the semi-further regularity condition (SFRC) and further regularity condition (SFRC) introduced in \cite{ref3}. The cited regularities are based on the properties of the epigraph of the conjugate functions. To proceed, we set
\begin{equation*}
	\Gamma =\operatorname{epi}(\lambda^{-1}F+f)^{*}, \ \ \Lambda =\underset{i=1}{\bigcup^{k}}\operatorname{epi}G_{i}^{*}, \ \ \Pi =\underset{j=1}{\bigcup^{p}}\operatorname{epi}g_{j}^{*}.
\end{equation*}
Based on the above sets, we introduce the  characteristic set ${\mathcal{R}}$ defined by
\begin{equation*}
{\mathcal{R}}:=\underset{x^{*}\in\operatorname{dom}{\varphi}^{*}}{\bigcap} \big( \Gamma + \operatorname{cone}\left( \Lambda \cup \Pi\right) +\operatorname{epi}\delta^{*}_{\mathbb{R}^{n}\times\mathbb{R}^{m} }-\left(\left( x^{*},0\right) ;{\varphi}^{*}(x^{*})\right) \big).
\end{equation*}
Considering the possible relationships between $\mathcal{R}$ and $\operatorname{epi}\left(\lambda^{-1}F+f-\varphi+\delta_{\mathcal{\mathcal{S}}}\right)^{*}$, we introduce the following regularity conditions.
\begin{definition}
	The family $\left((\lambda^{-1}F+f),\varphi,\delta_{\mathbb{R}^{n}\times\mathbb{R}^{m}},G_{i},g_{j}\right)$, $i=1,\dots,k$, $j=1,\dots,p $ is said to satisfy the semi-further regularity condition (SFRC) if
	\begin{equation}\label{sfrc-def} \mathcal{R}\cap\left(\{0_{\mathbb{R}^{n}\times\mathbb{R}^{m}}\}\times\mathbb{R}\right)\subseteq\operatorname{epi}\left(\lambda^{-1}F+f-\varphi+\delta_{\mathcal{S}}\right)^{*}\cap\left(\{0_{\mathbb{R}^{n}\times\mathbb{R}^{m}}\}\times\mathbb{R}\right).
	\end{equation}
If equality holds in $\left( \ref{sfrc-def}\right) $, we say that the further regularity condition (FRC) is satisfied.
\end{definition}

\begin{remark}
In general, neither
\begin{equation*}
	\mathcal{R}\subseteq\operatorname{epi}\left((\lambda^{-1}F+f)-\varphi+\delta_{\mathcal{S}}\right)^{*} \ \ \text{nor} \ \
	\operatorname{epi}\left((\lambda^{-1}F+f)-\varphi+\delta_{\mathcal{S}}\right)^{*} \subseteq\mathcal{R}
\end{equation*}
holds necessarily. But when $\varphi$ is convex, one gets
\begin{equation*}
	\mathcal{R}\subseteq\operatorname{epi}\left((\lambda^{-1}F+f)-\varphi+\delta_{\mathcal{S}}\right)^{*}.
\end{equation*}
Indeed, by Proposition $\ref{cl2}\ (ii)$, we have
\begin{equation*}
\begin{aligned}
\left((\lambda^{-1}F+f)-\varphi\right)^{*}&=\left(\inf\left(\lambda^{-1}F+f-x^{*}+{\varphi}^{*}(x^{*})\right)     \right)^{*}\\
&= \sup\left(\lambda^{-1}F+f-x^{*}+{\varphi}^{*}(x^{*})\right)^{*}.
\end{aligned}
\end{equation*}
This, together with $(\ref{eqr-1}) \ \text{and Proposition}\ \ref{cl2}\ (i)$,  implies that
\begin{equation*}
\begin{aligned}
\operatorname{epi}\left((\lambda^{-1}F+f)-\varphi\right)^{*}&=\underset{x^{*}\in\operatorname{dom}\varphi^{*}}{\bigcap}\operatorname{epi}\left(\lambda^{-1}F+f-x^{*}+{\varphi}^{*}(x^{*})\right)^{*}\\
&= \underset{x^{*}\in\operatorname{dom}\varphi^{*}}{\bigcap} \operatorname{epi}(\lambda^{-1}F+f)^{*}-(x^{*},{\varphi}^{*}(x^{*})).
\end{aligned}
\end{equation*}
Thus,
\begin{equation*}
\begin{array}{l}
\operatorname{epi}\left((\lambda^{-1}F+f)-\varphi\right)^{*}+\operatorname{cone}\left( \Lambda \cup\Pi\right) +\operatorname{epi}\delta^{*}_{\mathbb{R}^{n}\times\mathbb{R}^{m} }\\ = \underset{x^{*}\in\operatorname{dom}{\varphi}^{*}}{\displaystyle\bigcap}\left(\Gamma +\operatorname{cone}\left( \Lambda \cup\Pi\right)  +\operatorname{epi}\delta^{*}_{\mathbb{R}^{n}\times\mathbb{R}^{m} }-(x^{*},{\varphi}^{*}(x^{*}))\right)\\
= \mathcal{R}.
\end{array}
\end{equation*}
On the other hand, using the equality $(3.5)$ in \cite{ref6}, we obtain

\begin{equation*}
		\operatorname{epi}\left((\lambda^{-1}F+f)-\varphi\right)^{*}+\operatorname{cone}\left( \Lambda \cup\Pi\right) +\operatorname{epi}\delta^{*}_{\mathbb{R}^{n}\times\mathbb{R}^{m} } \subseteq \operatorname{epi}\left((\lambda^{-1}F+f)-\varphi+\delta_{\mathcal{S}}\right)^{*}.
\end{equation*}
Consequently,
 \begin{equation*}
 	\mathcal{R}\subseteq\operatorname{epi}\left((\lambda^{-1}F+f)-\varphi+\delta_{\mathcal{S}}\right)^{*}.
 \end{equation*}
\end{remark}

\begin{remark}\label{r.s.d}
It can be seen that the strong duality for \eqref{calm3} and \eqref{form-3} holds if and only if ${\mathcal V}\eqref{calm3}={\mathcal V}\eqref{form-3}$ and for any $x^{*} \in \operatorname{dom} \varphi^{*}$, we can find points $\left(z^{\ast},q^{\ast} \right)\in \operatorname{dom}(\lambda^{-1}F+f)^{*}$, $\left(u_{i}^{\ast},v_{i}^{\ast} \right)\in \operatorname{dom}G_{i}^{*}$, $i=1,\cdots ,k$, $\left(u_{j}^{\ast},v_{j}^{\ast} \right)\in \operatorname{dom}g_{j}^{*}$, $j=1,\cdots ,p$, $\mu_{i}\in\mathbb{R}_{+}$, $i=1,\cdots ,k$, $\nu_{j}\in \mathbb{R}_{+}$, $j=1,\cdots ,p$ and $\left(s^{\ast},t^{\ast}\right) \in \mathbb{R}^{n}\times\mathbb{R}^{m}$ such that we have 
	\begin{equation*}
		\varphi\left( x^{\ast}\right)-(\lambda^{-1}F+f)^{*}\left(z^{\ast},q^{\ast} \right)-{\displaystyle \sum^{k}_{i=1}} \ \mu_{i}G^{\ast}_{i}\left(u_{i}^{\ast},v_{i}^{\ast} \right)-{\displaystyle \sum^{p}_{j=1}} \ \nu_{j}g^{\ast}_{j}\left(u_{j}^{\ast},v_{j}^{\ast} \right)-\delta^{*}_{\mathbb{R}^{n}\times\mathbb{R}^{m} }\left(s^{\ast},t^{\ast}\right) \geq {\mathcal V}\eqref{form-3},
	\end{equation*}
with $s^{\ast} = x^{\ast}- z^{*}-\displaystyle\sum_{i=1}^{k} \mu_{i} u_{i}^{*}-\sum_{j=1}^{p} \nu_{j}u_{j}^{*}$, and $t^{\ast} = -q^{\ast}-\displaystyle\sum_{i=1}^{k} \mu_{i} v_{i}^{*}-\sum_{j=1}^{p} \nu_{j} v_{j}^{*}$.
\end{remark}

The following theorems provide weak and strong Fenchel–Lagrange duality results with regards to an optimization problem of the form \eqref{calm3}.
\begin{theorem}[Weak Fenchel–Lagrange duality]\label{w.d}
Let $\lambda>0$ and $F$, $f$, $G_{i}$, $g_{j}$, $i=1,\dots,k$, $j=1,\dots, p$ be proper convex functions. The family $\left((\lambda^{-1}F+f),\; \varphi,\delta_{\mathbb{R}^{n}\times\mathbb{R}^{m}},\; G_{i},\; g_{j}\right)$, $i=1,\dots,k$, $j=1,\dots,p$ satisfies (SFRC) if and only if the weak duality between \eqref{calm3} and \eqref{form-3} holds.
\end{theorem}
\begin{proof} Suppose that the weak duality between \eqref{calm3} and \eqref{form-3} does not hold. Then, 
\[
{\mathcal V}\eqref{calm3}<-\zeta <{\mathcal V}\eqref{form-3}
\]
for some $\zeta\in\mathbb{R}$. From \eqref{form-3}, for any $x^{*}\in\operatorname{dom}{\varphi}^{*}$, there exist $\alpha\in\mathbb{R}_{+}^{k}$, $\beta\in\mathbb{R}_{+}^{p}$, $\left( z^{*},q^{\ast}\right) \in\operatorname{dom} F^{*}\cap\operatorname{dom}f^{*}$, $(u_{i}^{*},v_{i}^{*})\in\operatorname{dom}G_{i}^{*}$, and $(u_{j}^{*},v_{j}^{*})\in\operatorname{dom}g_{j}^{*}$ such that

\begin{equation*}
	\varphi^{*}\left(x^{*}\right)-(\lambda^{-1}F+ f)^{*}\left( z^{*},q^{\ast}\right)-\displaystyle\sum_{i=1}^{k} \alpha_{i} G_{i}^{*}\left(u_{i}^{*},v_{i}^{*}\right)-\sum_{j=1}^{p} \beta_{j}g_{j}^{*}\left(u_{j}^{*},v_{j}^{*}\right)-\delta^{\ast}_{\mathbb{R}^{n}\times\mathbb{R}^{m}} \left( s^{\ast},t^{\ast}\right) \geqslant -\zeta ,
\end{equation*}
with $s^{\ast} = x^{\ast}- z^{*}-\displaystyle\sum_{i=1}^{k} \alpha_{i} u_{i}^{*}-\sum_{j=1}^{p} \beta_{j}u_{j}^{*}$, and $t^{\ast}=-q^{\ast}-\displaystyle\sum_{i=1}^{k} \alpha_{i} v_{i}^{*}-\sum_{j=1}^{p} \beta_{j} v_{j}^{*}$. Thus,
\begin{equation*}
	\delta^{\ast}_{\mathbb{R}^{n}\times\mathbb{R}^{m}} \left( s^{\ast},t^{\ast}\right) \leq \varphi^{*}\left(x^{*}\right)-(\lambda^{-1}F+ f)^{*}\left( z^{*},q^{\ast}\right)-\displaystyle\sum_{i=1}^{k} \alpha_{i} G_{i}^{*}\left(u_{i}^{*},v_{i}^{*}\right)-\sum_{j=1}^{p} \beta_{j}g_{j}^{*}\left(u_{j}^{*},v_{j}^{*}\right) +\zeta .
\end{equation*}
It follows that
\begin{equation*}
\left( \left( s^{\ast}, t^{*}\right) ; \varphi^{*}\left(x^{*}\right)-(\lambda^{-1}F+ f)^{*}\left( z^{*},q^{\ast}\right)-\displaystyle\sum_{i=1}^{k} \alpha_{i} G_{i}^{*}\left(u_{i}^{*},v_{i}^{*}\right)-\sum_{j=1}^{p} \beta_{j}g_{j}^{*}\left(u_{j}^{*},v_{j}^{*}\right) +\zeta\right) \in \text{epi} \delta^{\ast}_{\mathbb{R}^{n}\times\mathbb{R}^{m}}.
\end{equation*}
Or equivalently,
\begin{equation*}
\begin{array}{l}
\left( 0,\zeta\right) \in \left( -\left( s^{\ast}, t^{*}\right) ; -\varphi^{*}\left(x^{*}\right)+(\lambda^{-1}F+ f)^{*}\left( z^{*},q^{\ast}\right)\right.\\
\qquad \qquad \qquad \quad \left.+\displaystyle\sum_{i=1}^{k} \alpha_{i} G_{i}^{*}\left(u_{i}^{*},v_{i}^{*}\right)-\sum_{j=1}^{p} \beta_{j}g_{j}^{*}\left(u_{j}^{*},v_{j}^{*}\right) \right) + \text{epi} \delta^{\ast}_{\mathbb{R}^{n}\times\mathbb{R}^{m}}.
\end{array}
\end{equation*}
Writing $\left( s^{\ast}, t^{*}\right)$ as
\begin{equation*}
\left( s^{\ast}, t^{*}\right)=-\left( z^{*},q^{*}\right)-\left(\displaystyle\sum_{i=1}^{k} \alpha_{i}\left(  u_{i}^{*} ,v_{i}^{*}\right)\right)  -\left(\displaystyle\sum_{j=1}^{p} \beta_{i}\left(  u_{j}^{*} ,v_{j}^{*}\right)\right) +\left( x^{\ast},0\right),
\end{equation*}
one deduces that
\begin{equation*}
\begin{array}{lcl}
\left( 0;\zeta\right) & \in & \left( \left( z^{\ast},q^{\ast}\right);(\lambda^{-1}F+ f)^{*}\left(z^{*},q^{*}\right)\right) -\left( \left( x^{\ast},0\right);\varphi^{*}\left(x^{*}\right) \right) \\
& & +\displaystyle\sum_{i=1}^{k} \alpha_{i} \left( \left(  u_{i}^{*} ,v_{i}^{*}\right); G^{\ast}_{i}\left(  u_{i}^{*} ,v_{i}^{*}\right)\right)
  \displaystyle\sum_{j=1}^{p} \beta_{j} \left( \left(  u_{j}^{*} ,v_{j}^{*}\right); g^{\ast}_{j}\left(  u_{j}^{*} ,v_{j}^{*}\right)\right)
   + \operatorname{epi}\delta^{*}_{\mathbb{R}^{n}\times\mathbb{R}^{m} }.
\end{array}
\end{equation*}
Consequently,
\begin{equation*}
\left( 0;\zeta\right) \in	\big( \Gamma + \operatorname{cone}\left( \Lambda \cup \Pi\right) +\operatorname{epi}\delta^{*}_{\mathbb{R}^{n}\times\mathbb{R}^{m} }-\left(\left( x^{*},0\right) ;{\varphi}^{*}(x^{*})\right).
\end{equation*}
Since  $x^{\ast}\in \operatorname{dom} \varphi^{\ast}$ is arbitrary, we have that $(0,\zeta)\in \mathcal{R}$. Now, since the family\\ $\left((\lambda^{-1}F+f),\; \varphi,\; \delta_{\mathbb{R}^{n}\times\mathbb{R}^{m}}\right)$, $G_{i}$, $g_{j}$ for $i=1,\dots,k$, $j=1,\dots,p$ satisfies (SFRC), we have that
$$(0,\zeta)\in\operatorname{epi}\left(\lambda^{-1}F+f-\varphi+\delta_{\mathcal{S}}\right)^{*} .$$
Remarking that
\begin{equation*}
	{\mathcal V}\eqref{calm3}=\inf\{\left( \lambda^{-1}F+f\right)\left( x,y\right)  -\varphi\left( x\right) +\delta_{\mathcal S}\left( x,y\right) \}=-\left(  \lambda^{-1}F+f -\varphi +\delta_{\mathcal S} \right)^{\ast} \left( 0\right),
\end{equation*}
one can deduce
\begin{equation}\label{lemme-i}
(0,\zeta)\in\operatorname{epi}\left(\lambda^{-1}F+f-\varphi+\delta_{\mathcal{S}}\right)^{*} \  \ {\mathcal V}\eqref{calm3}\geqslant -\zeta ,
\end{equation}
Hence, ${\mathcal V}\eqref{calm3}\geqslant- \zeta$. This contradicts ${\mathcal V}\eqref{calm3}<-\zeta$. Thus, the weak duality between problem \eqref{calm3} and problem \eqref{form-3} holds.

Conversely, suppose the weak duality between \eqref{calm3} and \eqref{form-3} holds. Lets $
(0,\zeta)\in \mathcal{R}$ and $x^{\ast}\in \text{dom} \varphi$. By definition, we have
\begin{equation*}
	\left( x^{\ast},\zeta +\varphi\left( x^{\ast}\right)\right)  \in \operatorname{epi}(\lambda^{-1}F+f)^{*}+ \operatorname{cone}\left( \underset{i=1}{\bigcup^{k}}\operatorname{epi}G_{i}^{*} \cup \underset{j=1}{\bigcup^{p}}\operatorname{epi}g_{j}^{*}\right) +\operatorname{epi}\delta^{*}_{\mathbb{R}^{n}\times\mathbb{R}^{m} }.
\end{equation*}
Then there exist $\left( \left(z^{\ast},q^{\ast} \right);\gamma\right)\in \operatorname{epi}(\lambda^{-1}F+f)^{*}$,
\\$\left( \left(u_{i}^{\ast},v_{i}^{\ast} \right);\gamma_{i}\right)\in \operatorname{epi}G_{i}^{*}$, $i=1,\cdots ,k$, $\left( \left(u_{j}^{\ast},v_{j}^{\ast} \right);\gamma_{j}\right)\in \operatorname{epi}g_{j}^{*}$, $j=1,\cdots ,p$, $\mu_{i}\in\mathbb{R}_{+}$, $i=1,\cdots ,k$, $\nu_{j} \in\mathbb{R}_{+}$, $j=1,\cdots ,p$ and $\left( \left(s^{\ast},t^{\ast}\right) ;\sigma \right) \in \operatorname{epi}\delta^{*}_{\mathbb{R}^{n}\times\mathbb{R}^{m} }$ such that
\begin{equation*}
\left( \left( x^{\ast},0\right),\zeta +\varphi\left( x^{\ast}\right)\right)=\left( \left(z^{\ast},q^{\ast} \right);\gamma\right)+{\displaystyle \sum^{k}_{i=1}} \ \mu_{i}\left( \left(u_{i}^{\ast},v_{i}^{\ast} \right);\gamma_{i}\right)+{\displaystyle \sum^{p}_{j=1}} \ \nu_{i}\left( \left(u_{j}^{\ast},v_{j}^{\ast} \right);\gamma_{j}\right)+\left( \left(s^{\ast},t^{\ast}\right) ;\sigma \right).
\end{equation*}
Or equivalently,
\begin{equation}\label{expression1}
\begin{array}{rcl}
 \left( x^{\ast},0\right) & = & \left(z^{\ast},q^{\ast} \right)+{\displaystyle \sum^{k}_{i=1}} \ \mu_{i} \left(u_{i}^{\ast},v_{i}^{\ast} \right)+{\displaystyle \sum^{p}_{j=1}} \ \nu_{i}\left(u_{j}^{\ast},v_{j}^{\ast} \right)+\left(s^{\ast},t^{\ast}\right),\\
 \zeta +\varphi\left( x^{\ast}\right) & = &\gamma +{\displaystyle \sum^{k}_{i=1}} \ \mu_{i}\gamma_{i}+{\displaystyle \sum^{p}_{j=1}} \ \nu_{j}\gamma_{j}+\sigma .
\end{array}
\end{equation}
Using the definition of epigraph, we obtain
\begin{equation*}
	(\lambda^{-1}F+f)^{*}\left(z^{\ast},q^{\ast} \right)+{\displaystyle \sum^{k}_{i=1}} \ \mu_{i}G^{\ast}_{i}\left(u_{i}^{\ast},v_{i}^{\ast} \right)+{\displaystyle \sum^{p}_{j=1}} \ \nu_{j}g^{\ast}_{j}\left(u_{j}^{\ast},v_{j}^{\ast} \right)+\delta^{*}_{\mathbb{R}^{n}\times\mathbb{R}^{m} }\left(s^{\ast},t^{\ast}\right)\leq \zeta +\varphi\left( x^{\ast}\right).
\end{equation*}
Hence,
\begin{equation*}
	 -\zeta \leq \varphi\left( x^{\ast}\right)-(\lambda^{-1}F+f)^{*}\left(z^{\ast},q^{\ast} \right)-{\displaystyle \sum^{k}_{i=1}} \ \mu_{i}G^{\ast}_{i}\left(u_{i}^{\ast},v_{i}^{\ast} \right)-{\displaystyle \sum^{p}_{j=1}} \ \nu_{j}g^{\ast}_{j}\left(u_{j}^{\ast},v_{j}^{\ast} \right)-\delta^{*}_{\mathbb{R}^{n}\times\mathbb{R}^{m} }\left(s^{\ast},t^{\ast}\right).
\end{equation*}
Using the first equality in $\left( \ref{expression1}\right) $, we get ${\mathcal V}\eqref{form-3}\geqslant-\zeta$. Then ${\mathcal V}\eqref{calm3}\geqslant {\mathcal V}\eqref{form-3}\geqslant - \zeta$, which implies that $(0,\zeta)\in\operatorname{epi}\left(\lambda^{-1}F+f-\varphi+\delta_{\mathcal{S}}\right)^{*} $ by $\left( \ref{lemme-i}\right) $. 
 Then the family $\left( (\lambda^{-1}F+f),\varphi,\delta_{\mathbb{R}^{n}\times\mathbb{R}^{m}}, G_{i},g_{j}\right)$, $i=1,\dots,k$, $j=1,\dots,p $ satisfies (SFRC).
\end{proof}

\begin{theorem}[Strong Fenchel–Lagrange duality]\label{S.d}
	Let $\lambda>0$ and $F,f,G_{i},g_{j},i=1,\dots,k,j=1,\dots,p$ be proper convex functions. The family $\left((\lambda^{-1}F+f),\varphi,\delta_{\mathbb{R}^{n}\times\mathbb{R}^{m}},G_{i},g_{j}\right)$, $i=1,\dots,k$, $j=1,\dots,p $ satisfies (FRC) if and only if the strong duality between \eqref{calm3} and \eqref{form-3} holds.
\end{theorem}
\begin{proof}
	Suppose that the family $\left((\lambda^{-1}F+f),\varphi,\delta_{\mathbb{R}^{n}\times\mathbb{R}^{m}},G_{i},g_{j}\right)$, $i=1,\dots,k$, $j=1,\dots,p $ satisfies $(F R C)$. Then, it satisfies $(S F R C)$. By Theorem $\ref{w.d}$ , ${\mathcal V}\eqref{calm3} \geqslant {\mathcal V}\eqref{form-3}$. Now, If ${\mathcal V}\eqref{calm3}=-\infty$, then we get strong duality for \eqref{calm3} and \eqref{form-3} via weak duality. So, we assume that ${\mathcal V}\eqref{calm3}=-\zeta \in \mathbb{R}$. Then, from $\left( \ref{lemme-i}\right) $,
	$$
	(0, \zeta) \in \operatorname{epi}\left(\lambda^{-1}F+f-\varphi+\delta_{\mathcal{S}}\right)^{*}.
	$$
	Since the family $\left((\lambda^{-1}F+f),\varphi,\delta_{\mathbb{R}^{n}\times\mathbb{R}^{m}},G_{i},g_{j}\right)$, $i=1,\dots,k$, $j=1,\dots,p $ satisfies $(F R C)$, we have that $(0, \zeta) \in \mathcal
	R$. Thus, from the proof of Theorem \ref{w.d}, ${\mathcal V}\eqref{form-3} \geqslant-\zeta$ and for any $x^{*} \in \operatorname{dom} \varphi^{*}$, there exist $\left( \left(z^{\ast},q^{\ast} \right);\gamma\right)\in \operatorname{epi}(\lambda^{-1}F+f)^{*}$, $\left( \left(u_{i}^{\ast},v_{i}^{\ast} \right);\gamma_{i}\right)\in \operatorname{epi}G_{i}^{*}$, $i=1,\cdots ,k$, $\left( \left(u_{j}^{\ast},v_{j}^{\ast} \right);\gamma_{j}\right)\in \operatorname{epi}g_{j}^{*}$, $j=1,\cdots ,p$, $\mu_{i}\in\mathbb{R}_{+}$, $i=1,\cdots ,k$, $\nu_{j}\in\mathbb{R}_{+}$, $j=1,\cdots ,p$, and $\left( \left(s^{\ast},t^{\ast}\right) ;\sigma \right) \in \operatorname{epi}\delta^{*}_{\mathbb{R}^{n}\times\mathbb{R}^{m} }$ such that
	\begin{equation*}
		-\zeta \leq \varphi\left( x^{\ast}\right)-(\lambda^{-1}F+f)^{*}\left(z^{\ast},q^{\ast} \right)-{\displaystyle \sum^{k}_{i=1}} \ \mu_{i}G^{\ast}_{i}\left(u_{i}^{\ast},v_{i}^{\ast} \right)-{\displaystyle \sum^{p}_{j=1}} \ \nu_{j}g^{\ast}_{j}\left(u_{j}^{\ast},v_{j}^{\ast} \right)-\delta^{*}_{\mathbb{R}^{n}\times\mathbb{R}^{m} }\left(s^{\ast},t^{\ast}\right).
	\end{equation*}	
	By Remark $\ref{r.s.d}$, we obtain the strong duality between \eqref{calm3} and \eqref{form-3}.

	Conversely, assuming that strong duality between \eqref{calm3} and \eqref{form-3} holds, then by Remark $\ref{r.s.d}$, we have the equality ${\mathcal V}\eqref{calm3}={\mathcal V}\eqref{form-3}$ and for any $x^{*} \in \operatorname{dom} \varphi^{*}$, there exist $\left(z^{\ast},q^{\ast} \right)\in \operatorname{dom}(\lambda^{-1}F+f)^{*}$, $\left(u_{i}^{\ast},v_{i}^{\ast} \right)\in \operatorname{dom}G_{i}^{*}$, $i=1,\cdots ,k$, $\left(u_{j}^{\ast},v_{j}^{\ast} \right)\in \operatorname{dom}g_{j}^{*}$, $j=1,\cdots ,p$, $\mu_{i}\in\mathbb{R}_{+}$, $i=1,\cdots ,k$, $\nu_{j}\in\mathbb{R}_{+}$, $j=1,\cdots ,p$ and $\left(s^{\ast},t^{\ast}\right) \in \mathbb{R}^{n}\times\mathbb{R}^{m}$ such that we have
	\begin{equation*}
		\varphi\left( x^{\ast}\right)-(\lambda^{-1}F+f)^{*}\left(z^{\ast},q^{\ast} \right)-{\displaystyle \sum^{k}_{i=1}} \ \mu_{i}G^{\ast}_{i}\left(u_{i}^{\ast},v_{i}^{\ast} \right)-{\displaystyle \sum^{p}_{j=1}} \ \nu_{j}g^{\ast}_{j}\left(u_{j}^{\ast},v_{j}^{\ast} \right)-\delta^{*}_{\mathbb{R}^{n}\times\mathbb{R}^{m} }\left(s^{\ast},t^{\ast}\right) \geqslant {\mathcal V}\eqref{form-3},
	\end{equation*}
	with $s^{\ast}= x^{\ast}- z^{*}-\displaystyle\sum_{i=1}^{k} \alpha_{i} u_{i}^{*}-\sum_{j=1}^{p} \beta_{j}u_{j}^{*}$, and $t^{\ast}=-q^{\ast}-\displaystyle\sum_{i=1}^{k} \alpha_{i} v_{i}^{*}-\sum_{j=1}^{p} \beta_{j} v_{j}^{*}$.\\
Let $(0, \zeta) \in \operatorname{epi}\left(\lambda^{-1}F+f-\varphi+\delta_{\mathcal{S}}\right)^{*}$. By $\left( \ref{lemme-i}\right) $, ${\mathcal V}\eqref{calm3} \geqslant-\zeta$. Then, from the strong duality, ${\mathcal V}\eqref{form-3}={\mathcal V}\eqref{calm3} \geqslant-\zeta$. By the proof of Theorem \ref{w.d}, $(0, \zeta) \in \mathcal{R}$ and hence, the result.
\end{proof}

 \subsection{Duality under the closedness condition}
 In the above subsections, the Slater condition, (FRC), and (SFRC) are used. However, these conditions are
 often not satisfied for many problems in applications. To overcome this drawback, we use another constraint qualification, labeled as ``closedness condition'', which has been  developed and used for convex (infinite) optimization problems. In this case, a Toland-Fenchel-Lagrange duality is  investigated.  To proceed, let us assume that $\mathcal{S}\cap\operatorname{dom}(\lambda^{-1}F+ f)\neq\emptyset$ and $F$, $f$, and $\varphi$ are functions such that $F$ and $f$ are proper lower semicontinuous (l.s.c.) and convex, while $\varphi$ is proper convex and satisfies $\varphi^{**}=\varphi$. Here, $\mathcal{S}$ is the feasible set of \eqref{calm3} stated in $\left( \ref{eq4}\right)$.

 We denote ${\mathcal A}$ by
 \begin{center}
 	$\mathcal{A}:=\underset{\alpha\in \mathbb{R}_{+}^{k}}{\bigcup}\operatorname{epi}(\alpha G)^{*}\bigcup\underset{\beta\in \mathbb{R}_{+}^{p}}{\bigcup}\operatorname{epi}(\beta g)^{*}+\operatorname{epi}\delta^{*}_{\mathbb{R}^{n}\times\mathbb{R}^{m} },$
 \end{center}
 and introduce the  closedness condition defined by
 \begin{equation}\label{ccondi}\tag{CC}
 \operatorname{epi}(\lambda^{-1}F+ f)^{*}+\mathcal{A} \ \text{ is closed}.
 \end{equation}
 The closedness condition \eqref{ccondi} can be traced to Burachik and Jeyakumar \cite{ref5} and was used later to derive optimality conditions for convex cone constrained optimization problems. Different types of sufficient conditions for  \eqref{ccondi} are given in \cite{ref5}.

We are interested here in another form of duality called Tolland-Fenchel-Lagrange duality, which takes the following form:
\begin{equation}\label{cc-dual}\tag{D$_{TFL}^{\lambda}$}
\begin{array}{l}
  \inf_{(x^{*},y^{*})\in\mathbb{R}^{n}\times\mathbb{R}^{m}} \ \underset{ \underset{ \left( \alpha,\beta \right)\in\mathbb{R}_{+}^{k+p}}  {\left( z^{\ast},q^{\ast}\right)\in \mathbb{R}^{n+m} } }{\max}
\  \left\{{\varphi}^{*}(x^{*})+\delta_{\mathbb{R}^{m}}^{\ast}\left( y^{\ast}\right)\right.\\
 \qquad \qquad \qquad -\left.\left(\lambda^{-1}F+ f\right)^{*}\left(z^{\ast},q^{\ast} \right)- \left( \alpha G+ \beta g\right)^{\ast} \left( x^{\ast}-z^{\ast},y^{\ast}-q^{\ast}\right)\right\}.
\end{array}
\end{equation}

\begin{theorem} Consider the primal and the dual problem \eqref{calm3} and \eqref{cc-dual}. Suppose the constaint qualification \eqref{ccondi} holds. Then, the strong duality holds; i.e., ${\mathcal V} \eqref{calm3}={\mathcal V}\eqref{cc-dual}$.
\end{theorem}
\begin{proof}
By definition of a primal problem,
\begin{equation*}
{\mathcal V}\eqref{calm3}=\inf _{(x,y) \in\mathbb{R}^{n}\times\mathbb{R}^{m} }\left(\lambda^{-1}F+ f+\delta_{\mathcal{S}}-{\varphi}\right)(x,y).
\end{equation*}
Using the classical Toland dulality, we obtain
\begin{equation}\label{eq8}
{\mathcal V}(\mathcal{P}^{\lambda})=
\inf_{(x^{*},y^{*})\in\mathbb{R}^{n}\times\mathbb{R}^{m}}\left\{{\varphi}^{*}(x^{*})+\delta_{\mathbb{R}^{m}}^{\ast}\left( y^{\ast}\right) -\left(\lambda^{-1}F+f+\delta_{\mathcal{S}}\right)^{*}\left(x^{*},y^{*}\right)\right\}.
\end{equation}
Now, consider the function $\psi :\mathbb{R}^{n+m}\times\mathbb{R}^{n+m}\times\mathbb{R}^{k+p}\rightarrow \mathbb{R}\cup \{+\infty\}$ given by
\begin{equation*}
\psi\left(x,y,z,q,\nu ,\mu\right) =
\left\{
\begin{array}{lcl}
\left( \lambda^{-1}F+ f\right) \left(x+z,y+q\right), & \ & \text{if} \ \left( G\left( x,y\right) ,g\left( x,y\right)\right)  + \left( \nu ,\mu\right)  \in -\mathbb{R}_{+}^{k+p}, \\
+\infty & \ & \text{otherwise}.
\end{array}
\right.
\end{equation*}
Let $\left(x^{\ast},y^{\ast},z^{\ast},q^{\ast},\alpha,\beta \right) \in \mathbb{R}^{n+m}\times\mathbb{R}^{n+m}\times\mathbb{R}^{k+p}$. Using the same technique as in \cite{refrev0}, we calculate the conjugate of $\psi$ as
\begin{equation*}
	\psi^{\ast}\left(x^{\ast},y^{\ast},z^{\ast},q^{\ast},\alpha,\beta \right)=\left(\lambda^{-1}F+ f\right)^{*}\left(z^{\ast},q^{\ast}\right)+ \left( \alpha G+ \beta g\right)^{\ast} \left( x^{\ast}-z^{\ast},y^{\ast}-q^{\ast}\right) +\delta_{-\mathbb{R}_{+}^{k+p}}^{\ast}\left( \alpha,\beta\right).
\end{equation*}
Observing that $\delta_{-\mathbb{R}_{+}^{k+p}}^{\ast}\left( \alpha,\beta \right)=0$ if $\left( \alpha,\beta\right) \in \mathbb{R}_{+}^{k+p}$, one deduces from the last equality that
\begin{equation*}
\begin{array}{l}
  	\psi^{\ast}\left(x^{\ast},y^{\ast},z^{\ast},q^{\ast},\alpha,\beta  \right)=\\
	\qquad \qquad \left\{
	\begin{array}{lcl}
	\left(\lambda^{-1}F+ f\right)^{*}\left(z^{\ast},q^{\ast} \right)+ \left( \alpha G+ \beta g\right)^{\ast} \left( x^{\ast}-z^{\ast},y^{\ast}-q^{\ast}\right) & \ & \text{if} \  \left( \alpha,\beta\right)\in \mathbb{R}_{+}^{k+p},\\
	+\infty & \ & \text{otherwise}.
		\end{array}
		\right.
\end{array}
\end{equation*}
On the other hand, we have
\begin{equation}\label{cc-p1}
\begin{array}{lcl}
\left(\lambda^{-1}F+f+\delta_{\mathcal{S}}\right)^{*}\left(x^{*},y^{*}\right) & = & {\displaystyle \sup_{\left( x,y\right) \in {\mathcal S}}} \ \{ \langle \left( x,y\right) ,\left( x^{\ast},y^{\ast}\right) \rangle-\left(\lambda^{-1}F+f\right)\left( x,y\right)  \} \\
	& = &  {\displaystyle \sup_{\left( x,y\right) \in \mathbb{R}^{n}\times\mathbb{R}^{m}}} \ \{ \langle \left( x,y\right) ,\left( x^{\ast},y^{\ast}\right) \rangle-\psi\left( x,y,0,0,0,0\right)  \}.
\end{array}
\end{equation}
Since, $Pr_{\mathbb{R}^{n}\times\mathbb{R}^{m}\times \mathbb{R}}\left( \text{epi} \ \psi^{\ast}\right) =\operatorname{epi}(\lambda^{-1}F+ f)^{*}+\mathcal{A}$, it follows from \cite[Theorem 3.1]{Wu100} that for all $\left( x^{\ast},y^{\ast}\right) \in \mathbb{R}^{n}\times\mathbb{R}^{m}$,
\begin{equation}\label{cc-p2}
\begin{array}{l}
{\displaystyle \sup_{\left( x,y\right) \in \mathbb{R}^{n}\times\mathbb{R}^{m}}} \ \{ \langle \left( x,y\right) ,\left( x^{\ast},y^{\ast}\right) \rangle-\psi\left( x,y,0,0,0,0\right)  \}\\
\qquad \qquad \qquad \qquad \qquad \quad  ={\displaystyle \min_{\left( z^{\ast},q^{\ast}\right)\in \mathbb{R}^{n+m}, \left( \alpha,\beta \right)\in\mathbb{R}^{k+p} }} \ \psi^{\ast}\left(x^{\ast},y^{\ast},z^{\ast},q^{\ast},\alpha,\beta \right).
\end{array}
\end{equation}
Combining $\left( \ref{cc-p1}\right)$ and $\left( \ref{cc-p2}\right)$, we arrive at
\begin{equation*}
\begin{array}{l}
  \left(\lambda^{-1}F+f+\delta_{\mathcal{S}}\right)^{*}\left(x^{*},y^{*}\right) \\
  \qquad \qquad \qquad = {\displaystyle \min_{\left( z^{\ast},q^{\ast}\right)\in \mathbb{R}^{n+m}, \left( \alpha^{\ast},\beta^{\ast} \right)\in\mathbb{R}_{+}^{k+p} }} \ \left(\lambda^{-1}F+ f\right)^{*}\left(z^{\ast},q^{\ast} \right)+ \left( \alpha^{\ast} G+ \beta^{\ast} g\right)^{\ast} \left( x^{\ast}-z^{\ast},y^{\ast}-q^{\ast}\right).
\end{array}	
\end{equation*}
Inserting the latter equality in $\left( \ref{eq8}\right)$ we obtain
\begin{equation*}
\begin{array}{l}
  {\mathcal V}\eqref{calm3}=
\inf_{(x^{*},y^{*})\in\mathbb{R}^{n}\times\mathbb{R}^{m}} \
\underset{ \underset{ \left( \alpha,\beta \right)\in\mathbb{R}_{+}^{k+p}}  {\left( z^{\ast},q^{\ast}\right)\in \mathbb{R}^{n+m} } }{\max}
 \  \left\{{\varphi}^{*}(x^{*})+\delta_{\mathbb{R}^{m}}^{\ast}\left( y^{\ast}\right)\right. \\
 \qquad \qquad \qquad -\left.\left(\lambda^{-1}F+ f\right)^{*}\left(z^{\ast},q^{\ast} \right)- \left( \alpha G+ \beta g\right)^{\ast} \left( x^{\ast}-z^{\ast},y^{\ast}-q^{\ast}\right)\right\}.
\end{array}
\end{equation*}
Hence,
${\mathcal V}\eqref{calm3}={\mathcal V}\eqref{cc-dual}.$
\end{proof}

\section{Duality without partial calmness}
\subsection{Under a generalized Slater condition}
In this section, we do not assume that \eqref{calm0} has optimal solutions, also no convexity assumption is imposed. Let $\epsilon > 0$ and consider the following reguralized problem of \eqref{calm0}
\begin{equation}\label{reg-bil}\tag{P$_{\epsilon}$}
{\displaystyle \min_{x,y}} \ F\left( x,y\right) \ \text{s.t.} \ G_{i}\left(x, y\right) \leq 0, \; \; i\in I_{k} := \{1,\cdots ,k\}, \; \; y\in {\mathcal S}_{\epsilon}\left( x\right),
\end{equation}
where for each $x\in \mathbb{R}^{n}$, ${\mathcal S}_{\epsilon}\left(x\right)$ is the set of $\epsilon$-approximate optimal
solutions of the lower level problem \eqref{llp-1}; that is
\begin{equation*}
{\mathcal S}_{\epsilon}\left( x\right) =\{y\in K\left( x\right)| \;\; f\left( x,y\right) -\varphi\left( x\right) \leq \epsilon  \}.
\end{equation*}
Then, it is clear that problem \eqref{reg-bil} can be reformulated as the following single-level optimization problem involving the optimal value function:
\begin{equation}\label{w-p-c-P}\tag{LLVF$_{\epsilon}$}
\left\{\begin{array}{ll}
{\displaystyle \min_{x, y}} \ F(x, y) \\[1ex]
G_{i}(x,y) \leq 0 \quad & \forall i \in I_{k}=\{1,\dots,k\}, \\
g_{j}(x, y) \leq 0 & \forall j \in I_{p}=\{1,\dots,p\}, \\
f(x, y)-\varphi(x) \leq \epsilon. \\
\end{array}\right.
\end{equation}
Next, we recall the relation between any accumulation point of a sequence of regularized
optimal solutions (i.e., of problem \eqref{w-p-c-P}) and the optimal solution of \eqref{calm0}. To this end, let  $\left(\epsilon_{k}\right)_{k}$ be a sequence of positive scalar such that $\epsilon_{k}\searrow 0$, we denote problem $\left(\mbox{LLVF}_{\epsilon_{k}}\right)$ by $\left(\mbox{LLVF}_{k}\right)$. We need the following Slater-type condition for the lower level problem:
\begin{assum}\label{regul}
	For any $x\in \mathbb{R}^{n}$ satisfying the upper-level inequality constraints, and any nonempty subset $J_{p}$ of $I_{p}$, there exist $y\in\mathbb{R}^{m}$ such that $g_{j}\left( x,y\right) < 0$ for all $j\in J_{p}$.
\end{assum}
\begin{theorem}\cite{refextend}
	Let  $\left(  \epsilon_{k}\right)_{k}$ a sequence of positive scalar such that $\epsilon_{k}\searrow 0$ and $\left( x_{k},y_{k}\right)_{k}$ be a sequence of optimal solutions of the regularized problems $\left(\mbox{LLVF}_{k}\right)$ with $k\in\mathbb{N}$. Suppose that
	\begin{enumerate}
		\item ${\mathcal C} :=\{\left( x,y\right) \in\mathbb{R}^{n}\times\mathbb{R}^{m}|\;\; G_{i}\left( x,y\right) \leq 0, \; \; i\in I_{k}\}$ is a compact set;
		\item $\left( \bar{x},\bar{y}\right)$ is an accumulation point of the sequence $\left( x_{k},y_{k}\right)_{k}$;
		\item Assumption \ref{regul}  holds.
	\end{enumerate}
	Then, $\left( \bar{x},\bar{y}\right)$ is an optimal solution of problem \eqref{llvf}.
\end{theorem}
Subsequently, based on the study given by Fabin, Fernando, and Cristian  \cite{fabin-12}, we give a
formulation of dual problem of \eqref{w-p-c-P}, then we provide  a new complete characterization of strong duality.
To proceed, for any $\epsilon > 0$, we consider the Lagrangian
\begin{equation*}
L_{\epsilon}\left(\varsigma ,\alpha ,\beta ,\gamma , x,y \right) := \varsigma F(x,y)+\gamma\left(f(x, y) - \varphi(x) - \epsilon\right)
 +{\displaystyle \sum_{i=1}^{k}} \ \alpha_{i} G_{i}(x,y)+ {\displaystyle \sum_{j=1}^{p}} \ \beta_{j} g_{j}(x,y)
\end{equation*}
and define the dual problem \eqref{dual-w-p-c-P} of \eqref{w-p-c-P} as
\begin{equation}\label{dual-w-p-c-P}\tag{D$_{\epsilon}$}
\underset{ \underset{ \gamma \in\mathbb{R}}  {\left( \alpha ,\beta\right) \in \mathbb{R}_{+}^{k+p}}} {\sup} \ {\displaystyle \inf_{\left( x,y\right) \in\mathbb{R}^{n}\times\mathbb{R}^{m}}} \ L_{\epsilon}\left(1 ,\alpha ,\beta ,\gamma , x,y \right).
\end{equation}
It can be seen that for all $\alpha\in\mathbb{R}_{+}^{k}$, $\beta\in\mathbb{R}_{+}^{p}$, and $\gamma\in\mathbb{R}$, we have
\begin{equation*}
	{\displaystyle \inf_{\left( x,y\right) \in\mathbb{R}^{n}\times\mathbb{R}^{m}}} \ L_{\epsilon}\left(\varsigma ,\alpha ,\beta ,\gamma , x,y \right)\;\; \leq \;\; \varsigma  {\mathcal V}\left( {\mathcal P}_{\epsilon}\right), \;\;\, \forall \varsigma \geq 0,
\end{equation*}
which corresponds to a weak duality condition. To provide the strong duality result, we introduce some regularity conditions. To proceed, define
\begin{equation*}
	\Psi_{\epsilon} \left( \mathbb{R}^{n}\times\mathbb{R}^{m}\right) := \left\{\left(F(x,y),\, f(x,y)-\varphi\left(x\right)-\epsilon,\, G(x,y),\, g(x,y) \right) \in\mathbb{R}^{2+k+p}:\;\, (x,y)\in \mathbb{R}^{n}\times\mathbb{R}^{m}\right\}
\end{equation*}
and consider the following assumptions:
\begin{assum} \label{ass-s-1}It holds  that
\[
\text{cone} \ \bigg ( int \bigg [ \text{co} \left( \Psi_{\epsilon} \left( \mathbb{R}^{n}\times\mathbb{R}^{m}\right) \right)-{\mathcal V}\eqref{w-p-c-P} \left( 1,0_{\mathbb{R}^{1+k+p}}\right)  + \left( \mathbb{R}_{+}\times\mathbb{R}^{1+k+p}_+\right) \bigg] \bigg) \;\; \mbox{ is pointed.}
\]
\end{assum}
\begin{assum} \label{ass-s-2}
	$\left( 0,0\right) \notin int \bigg ( \text{co} \left( \Psi_{\epsilon} \left( \mathbb{R}^{n}\times\mathbb{R}^{m}\right) \right)-{\mathcal V}\eqref{w-p-c-P} \left( 1,0_{\mathbb{R}^{1+k+p}}\right)  + \left( \mathbb{R}_{+}\times\mathbb{R}^{1+k+p}_+\right) \bigg)$.
\end{assum}
\begin{theorem}\label{thm-w-p-c-s}
	Let $\epsilon > 0$. Suppose that ${\mathcal V}\eqref{w-p-c-P}$ is finite and that
\begin{equation}\label{EmpE}
int \bigg ( \text{co} \left( \Psi_{\epsilon} \left( \mathbb{R}^{n}\times\mathbb{R}^{m}\right) \right)+ \left( \mathbb{R}_{+}\times\mathbb{R}^{1+k+p}_+\right) \bigg) \neq\emptyset.
\end{equation}
Then, for each $\epsilon > 0$, Assumption \ref{ass-s-1} (resp. Assumption \ref{ass-s-2}) holds if and only if there exists a vector  $\left(\bar{\varsigma}, \bar{\gamma},\bar{\alpha},\bar{\beta}\right)\neq 0 $ such that
	\begin{equation*}
		{\displaystyle \inf_{\left( x,y\right) \in\mathbb{R}^{n}\times\mathbb{R}^{m}}} \ L_{\epsilon}\left( \bar{\varsigma},\bar{\alpha} ,\bar{\beta} ,\bar{\gamma} , x,y \right)= \bar{\varsigma}  {\mathcal V}\eqref{w-p-c-P}.
	\end{equation*}
\end{theorem}
\begin{proof}
	The result follows from \cite[Theorem 3.1]{fabin-12}.
\end{proof}

Next, we provide a strong duality result when $\bar{\varsigma} > 0$. For this result,  we need the assumption
\begin{equation}\label{cq-w-p-c}
\overline{cone} \, \left(\left(f-\varphi -\epsilon, \; G,\; g\right) \left(\mathbb{R}^{n}\times\mathbb{R}^{m}\right)  + \mathbb{R}_{+}^{1+k+p}\right)= \mathbb{R}^{1+k+p}.
\end{equation}
The qualification condition \eqref{cq-w-p-c} is a generalized form of the Slater condition for problem \eqref{w-p-c-P}. Note that thanks to our regularization, condition \eqref{cq-w-p-c} can be fulfilled.
\begin{theorem}Let $\epsilon > 0$ and suppose that ${\mathcal V}\eqref{w-p-c-P}$ is finite and condition \eqref{EmpE} holds.
	Furthermore, if condition \eqref{cq-w-p-c} is satisfied, then strong duality between \eqref{w-p-c-P} and \eqref{dual-w-p-c-P} holds; i.e.,  ${\mathcal V}\eqref{w-p-c-P}={\mathcal V}\eqref{dual-w-p-c-P}$. Moreover, there exists a Lagrange multiplier vector $\left(\bar{\gamma},\bar{\alpha},\bar{\beta}\right)\in \mathbb{R}_{+}\times\mathbb{R}^{k}_{+}\times\mathbb{R}^{p}_{+}$ such that we have
	\begin{equation*}
		{\mathcal V}\eqref{w-p-c-P}=	{\displaystyle \inf_{\left( x,y\right) \in\mathbb{R}^{n}\times\mathbb{R}^{m}}} \ L_{\epsilon}\left(1,\bar{\alpha} ,\bar{\beta} ,\bar{\gamma} , x,y \right).
	\end{equation*}
\end{theorem}
\begin{proof}
	The result follows from \cite[Corollary 3.2]{fabin-12}.
\end{proof}

\subsection{Without a Slater condition}
To close this section, we provide a new characterization of strong duality for bilevel optimization problem, covering
situations where a Slater-type condition may fail. To proceed, consider the  special bilevel program
\begin{equation}\label{spe-bilevel}\tag{P$_{s}$}
{\displaystyle \min_{x,y}} \ F\left( x,y\right) \ \text{s.t.} \ x\in X, y\in {\mathcal S}\left( x\right),
\end{equation}
where, $X\subset \mathbb{R}^{n}$ and for each $x\in X$, ${\mathcal S}\left( x\right)$ is the set of optimal solutions of the lower level problem \eqref{llp-1}, which in this case is given
\begin{equation*}
{\mathcal S}\left( x\right) =\left\{y\in Y|\; f\left(x,y\right) -\varphi\left( x\right) \leq 0\right\},
\end{equation*}
where the lower level feasible set is $Y\subset \mathbb{R}^{m}$. As before, problem \eqref{spe-bilevel} can take the form
\begin{equation}\label{spe-bilevel-llvf}\tag{LLVF$_{s}$}
\left\{
\begin{array}{l}
{\displaystyle \min_{x,y}} \ F\left( x,y\right) \\
\left( x,y\right) \in X\times Y, \\
f\left( x,y\right) -\varphi\left( x\right)  \leq 0.
\end{array}
\right.
\end{equation}
With  $L\left(1 ,\gamma , x,y \right) := F\left( x,y\right) +\gamma\left( f\left( x,y\right) -\varphi\left( x\right) \right)$, the dual of problem  \eqref{dual-spe-bilevel-llvf} is given by
\begin{equation}\label{dual-spe-bilevel-llvf}\tag{D$_{s}$}
{\displaystyle \sup_{\gamma \in\mathbb{R}_+}} \ {\displaystyle \inf_{\left( x,y\right) \in X\times Y }} \ L\left(1 ,\gamma , x,y \right).
\end{equation}

To provide our result, denote the feasible set of problem \eqref{spe-bilevel-llvf} by
\begin{equation*}
	{\mathcal F} := \left\{\left( x,y\right) \in X\times Y|\;\, f\left( x,y\right) -\varphi\left( x\right)  \leq 0 \right\}.
\end{equation*}
The set of optimal solutions of problem \eqref{spe-bilevel-llvf} is given as
\begin{equation*}
	{\mathcal O} := \arg\min\left\{F\left( x,y\right) |\left( x,y\right) \in {\mathcal F} \right\}
\end{equation*}
and
\begin{equation*}
	\begin{array}{lcl}
		{\mathcal F}^{-} & := & \{\left( x,y\right) \in X\times Y |\;\, f\left( x,y\right) -\varphi\left( x\right)  < 0 \}, \\
		{\mathcal F}^{=} & := & \{\left( x,y\right) \in X\times Y |\;\, f\left( x,y\right) -\varphi\left( x\right)  = 0 \}, \\
		{\mathcal F}^{+} & := & \{\left( x,y\right) \in X\times Y |\;\, f\left( x,y\right) -\varphi\left( x\right)  > 0 \}.
	\end{array}
\end{equation*}
Analogously, we introduce the following sets
\begin{equation*}
	\begin{array}{lcl}
		\Xi^{-} & := & \{\left( x,y\right) \in X\times Y |\;\,F\left( x,y\right)  < {\mathcal V} \eqref{spe-bilevel-llvf} \},  \\
		\Xi^{=} & := & \{\left( x,y\right) \in X\times Y |\;\,F\left( x,y\right)  = {\mathcal V} \eqref{spe-bilevel-llvf} \}, \\
		\Xi^{+} & := & \{\left( x,y\right) \in X\times Y |\;\,F\left( x,y\right)  > {\mathcal V} \eqref{spe-bilevel-llvf} \}.
	\end{array}
\end{equation*}

\begin{remark} Based on the structure of bilevel optimization problems, one can deduce that  we have ${\mathcal F}^{-}={\mathcal F}^{+}=\emptyset$.
\end{remark}
Also, due to the lack of the Slater condition, the strong duality requires a closed-cone constraint qualification. To proceed, set
\begin{equation*}
\Psi \left(X\times Y\right) =\left\{\left(F\left( x,y\right), \, f\left( x,y\right)-\varphi\left( x\right) \right) \in\mathbb{R}^{2}\left| (x,y)\in X\times Y\right.\right\},
\end{equation*}
and consider the following assumption
\begin{assum} \label{ass-s-1-2}
$\text{cone} \ \left( \text{co} \left( \Psi \left(X\times Y\right) \right)-{\mathcal V} \eqref{spe-bilevel-llvf} \left( 1,0\right)  + int \left( \mathbb{R}^{2}_{+}\right)  \right)$ is pointed.
\end{assum}
Based on the above construction  we characterize  Assumption \ref{ass-s-1-2}.
\begin{proposition}\label{prop-final}
	Consider the problem \eqref{spe-bilevel-llvf}. Suppose that ${\mathcal F}\neq \emptyset$ and that ${\mathcal V}\eqref{spe-bilevel-llvf}$ is finite. The following assertions hold:
	\begin{enumerate}
		\item Suppose that ${\mathcal O}\neq \emptyset$. Then, $\text{cone} \ \bigg ( \text{co} \left( \Psi \left(X\times Y\right) \right)-{\mathcal V} \eqref{spe-bilevel-llvf} \left( 1,0\right)  + int \left( \mathbb{R}^{2}_{+}\right)  \bigg)$ is pointed if and only if ${\mathcal O}\cap {\mathcal F}^{=}\neq \emptyset$ and $\Xi^{+}\cap {\mathcal F}^{=}\neq \emptyset$.
			\item Assume that ${\mathcal O}= \emptyset$. Then, $\text{cone} \ \bigg ( \text{co} \left( \Psi \left(X\times Y\right) \right)-{\mathcal V} \eqref{spe-bilevel-llvf} \left( 1,0\right)  + int \left( \mathbb{R}^{2}_{+}\right)  \bigg)$ is pointed if and only if $\Xi^{+}\cap {\mathcal F}^{=}\neq \emptyset$.
	\end{enumerate}
\end{proposition}
The following theorem provides a complete characterization of strong duality for our bilevel program.
\begin{theorem}\label{th-final}
	Consider problem \eqref{spe-bilevel-llvf}. Assume that ${\mathcal F}\neq \emptyset$ and that ${\mathcal V} \eqref{spe-bilevel-llvf}$ is finite. The following assertions are equivalent.
	\begin{enumerate}
		\item Strong duality holds; i.e.,  ${\mathcal V}\eqref{spe-bilevel-llvf}={\mathcal V}\eqref{dual-spe-bilevel-llvf}$;
		\item $\text{cone} \ \bigg ( \text{co} \left( \Psi \left(X\times Y\right) \right)-{\mathcal V}\eqref{spe-bilevel-llvf} \left( 1,0\right)  + int \left( \mathbb{R}^{2}_{+}\right)  \bigg)$ is pointed. 
	\end{enumerate}
\end{theorem}
\begin{proof}
It can be deduced from Proposition \ref{prop-final}; cf. \cite[Theorem 4.1 and Corollary 4.2]{fabin-12}.
\end{proof}
\section{Conclusion}
In this paper, we have been given a duality approach for problem \eqref{calm0}. Since the bilevel programming
problem \eqref{calm0} does not satisfy the Slater condition, which is the key apparatus to establish strong duality, we have first provided an overview of some standard duality results to show that they are not applicable to our problem \eqref{calm0}. Afterwards, the concept of partial calmness is used to establish weak and strong duality results for problem \eqref{calm0}. In particular, Langrange, Fenchel-Lagrange, and Toland-Fenchel-Lagrange duality concepts were investigated for this problem, under some suitable conditions. For the first class of results, we introduce a regularization
for problem \eqref{calm0}. Then, we establish sufficient conditions ensuring strong duality for the obtained regularized problem \eqref{calm0} under a generalized Slater-type condition without the convexity assumptions and without. Finally, we have provided a strong duality result for a geometrically constrained  bilevel optimization problem without the Slater-condition.


\end{document}